\documentclass[10pt]{article}
\usepackage{epsfig,amsmath,amssymb}
\usepackage{amssymb,amsmath,amsthm}
\usepackage{color}
\usepackage{caption}
\usepackage{latexsym}
\usepackage{multicol}
\usepackage[centering]{geometry} 
\usepackage{wrapfig}
\usepackage{anysize}
\usepackage{appendix}
\usepackage{float}

\makeatletter      
\renewcommand{\ps@plain}{%
     \renewcommand{\@oddhead}{\textrm{ }\hfil\textrm{\thepage}}%
     \renewcommand{\@evenhead}{\@oddhead}%
     \renewcommand{\@oddfoot}{}
     \renewcommand{\@evenfoot}{\@oddfoot}}

\newcommand{\rf}[1]{#1}
\newcommand{\rs}[1]{#1}

\begin{document}

\newtheorem{remark}{Remark}
\newtheorem{definition}{Definition}
\newtheorem{theorem}{Theorem}
\newtheorem{proposition}{Proposition}

\title{Well-balanced high-order finite difference methods for systems of balance laws}

\author{ Carlos Par\'es, Carlos Par\'es-Pulido\\
  University of M\'alaga (Spain), ETH Z\"urich (Switzerland)}
  
  \date{\today}
  
  \maketitle

\begin{abstract}
In this paper, high order well-balanced finite difference weighted
essentially non-oscillatory methods to solve general systems of balance laws are presented. Two different families are introduced: while the methods in the first one preserve every stationary solution, those in the second family only preserve a given set of stationary solutions that depend on some parameters. The accuracy, well-balancedness, and conservation properties of the methods are discussed, as well as their application to systems with singular source terms. The strategy is applied to derive third and fifth order well-balanced methods for a linear scalar balance law, Burgers' equation with a nonlinear source term, and for the shallow water model. In particular, numerical methods that preserve every stationary solution or only water at rest equilibria are derived for the latter.
\end{abstract}

\medskip\noindent
{\bf Acknowledgments.} This research  has been partially supported by the Spanish Government and FEDER through the Research project RTI2018-096064-B-C21.

\section{Introduction}

We consider 1d systems of balance laws of the form
\begin{equation}\label{sbl}
U_t + F(U)_x = S(U) H_x,
\end{equation}
where $U(x,t)$ takes value in $\Omega \subset \mathbb{R}^N$, $F: \Omega  \to \mathbb{R}^N$ is the flux function;
$S: \Omega \to \mathbb{R}^N$; and $H$ is a known function from $\mathbb{R} \to \mathbb{R}$ (possibly the identity function $H(x) = x$). The system is supposed to be hyperbolic, i.e. the Jacobian $J(U)$ of the flux function is assumed to have $N$ different real eigenvalues. 

Systems of the form \eqref{sbl} have non trivial stationary solutions that satisfy the ODE system:
\begin{equation}\label{sblst}
F(U)_x = S(U)H_x.
\end{equation}

PDE systems of this form appear in many fluid models in different contexts:
shallow water models, multiphase flow models, gas dynamic, elastic wave equations, etc.
\rs{In particular, the shallow water system corresponds to \eqref{sbl} with the choices:
\begin{equation}\label{sw}
U = \left[ \begin{array}{c} h \\ q \end{array} \right], \quad F(U) = \left[ \begin{array}{c} q \\\displaystyle \frac{q^2}{h} + \frac{g}{2}h^2 \end{array} \right], 
\quad S(U) = \left[ \begin{array}{c} 0 \\ gh \end{array} \right]. 
\end{equation}
The variable $x$ makes reference to the axis of the channel and $t$ is time; $q(x,t)$ and $h(x,t)$ represent the mass-flow and the thickness,
respectively; $g$, the acceleration due to gravity;  $H(x)$, the depth measured from a fixed level of reference;  $q(x,t)=h(x,t)u(x,t),$ with $u$ the depth averaged horizontal velocity. The eigenvalues of the Jacobian matrix $J(U)$ of the flux function $f(U)$ are the following:
$$\lambda_1=u-\sqrt{gh}, \quad \lambda_2=u+\sqrt{gh}.$$
The Froude number, given by 
\begin{equation}
Fr(U)= \displaystyle \frac{|u|}{gh},
\end{equation}
indicates the flow regime: subcritical ($Fr<1$), critical ($Fr=1$) or supercritical ($Fr>1$). }

\rs{The stationary solutions of this system are  implicitly given by 
\begin{equation}\label{sw-ss}
q = C_1, \quad \frac{1}{2}\frac{q^2}{h^2} + gh - g H = C_2,
\end{equation}
where $C_i$,  $i=1,2$ are arbitrary constants. In particular, water at rest equilibria are the 1-parameter family of stationary solutions corresponding to $C_1 = 0$, i.e.
\begin{equation}\label{sw-war}
q = 0, \quad h - H = \bar \eta,
\end{equation}
where $\bar \eta$ is a constant, corresponding to the vertical coordinate of the elevation of the unperturbed surface of the water.}

\rs{The objective of well balanced schemes is to preserve exactly or with enhanced accuracy some of the steady state solutions.   In the context of shallow water
equations Berm\'udez and V\'azquez-Cend\'on introduced in
\cite{BV94} the condition called {\it C-property}: a
scheme is said to satisfy this condition if it preserves the water at rest solutions \eqref{sw-war}. 
Since then, many different numerical methods that satisfy this property have been introduced in the literature: see \cite{Bouchut04}, \cite{Xing17} and their references. 
In the framework of finite difference methods,  high-order schemes that satisfy the C-property were introduced in 
 \cite{Caselles09} and  \cite{XingShu1}. The former were based  in a technique consisting on a formal transformation of \eqref{sbl} into a conservative system through the definition of a ‘combined flux’ formed by the  flux  $F$  and a primitive of the source term: see \cite{Gascon},  \cite{Donat10}. The latter relied on the expression of the source term as a function of variables that are constants for the stationary solutions to be preserved: see \cite{XingShu2}.} 

\rs{In \cite{CastroPardoPares} a first order finite volume method that preserves all the stationary solutions \eqref{sw-ss} was introduced. The method was based on a generalization of the Hydrostatic Reconstruction technique introduced in \cite{Audusse04} to preserve water at rest solutions. Since then, different first and high order finite volume and DG methods that preserve all the stationary solutions \eqref{sw-ss}
have been described in the literature: see \cite{Berthon}, \cite{BouchutMorales}, \cite{Cheng}, \cite{lopez2013}, \cite{Noelle06}, \cite{Noelle07},  \cite{Russo}, \cite{Xing14}, \dots Nevertheless,  to the best of our knowledge, high order finite difference methods with this enhanced well-balanced property have not been described so far.}

\rs{The design of high-order well-balanced numerical methods for variants of the shallow water model (with friction, Coriolis term, RIPA model, shallow water system in spherical coordinates, etc.) as well as to other systems of balance laws (Euler equations with gravity, models for the blood flow in vessels, etc.) is a very active front of research: see, for instance,
\cite{Berberich19},  \cite{SWesfera}, \cite{CastroPares2020},
\cite{Chandrashekar15}, \cite{Chandrashekar17}, \cite{Chertock}, \cite{Chertock18}, \cite{ChertockRIPA},
\cite{Gaburro17}, \cite{Gaburro18-2},
\cite{Grosheintz19}, 
\cite{Kappeli14}, \cite{Lukacova}, \cite{venas},  \cite{RIPA} ...}
 
\rs{ We focus here on the development of well-balanced high-order  finite difference methods for general systems of balance laws that preserve all the stationary solutions. 
The strategy is based on a simple idea:  let $U_i$ be the numerical approximation of the solution $U(x_i, t)$ at the node
 $x_i$ at time $t$ and let  $U^*_i$ be the stationary solution satisfying the Cauchy problem:
 \begin{equation}\label{mot1}
 \begin{cases}
      & \displaystyle F(U^*_i)_x = S(U^*_i)H_x, \\[0.5em]
      & \displaystyle U^*_i(x_i) = U_i.
 \end{cases}
 \end{equation}
 Then, if $U^*_i$ can be found, one has trivially
 \begin{equation}\label{mot2}
 S(U_i) H_x (x_i)  = S(U^*_i(x_i)) H_x (x_i) = F(U^*_i(x_i))_x.
 \end{equation}
 Therefore, the source term can be numerically computed with high order accuracy by approaching the derivative of the function $F(U^*_i(x))$ at $x = x_i$ using a reconstruction operator. Of course, the main difficulty comes from the computation of the stationary solution $U^*_i$ at every node at every time step. As it will be shown, this technique can be easily adapted to the design of methods that preserve a prescribed set of stationary solutions: instead of solving the Cauchy problem \eqref{mot1},
 one stationary solution is chosen out of this set, whose value at $x_i$ is closest to $U_i$, in a sense to be determined. 
 The application of these techniques to the shallow water system will give numerical methods that preserve (a) all  the stationary solutions \eqref{sw-ss}, (b) the water at rest stationary solutions \eqref{sw-war} or (c) only the stationary solution corresponding to a particular choice of the constants $C_i$, $i = 1,2$, which is equivalent to fixing the mass-flow and the total energy of the equilibrium to be preserved.}
 
 This strategy has been inspired by the concept of well-balanced reconstruction introduced in  \cite{sinum2008} to develop well-balanced high-order finite volume numerical methods: see \cite{CastroPares2020} for a recent follow-up. In that case, \rs{if $U_i$ is the approximation of the cell-average  of the solution at the $i$th cell $I_i$ at time $t$, the stationary solution  $U^*_i$ whose cell-average is $U_i$ has to be found, i.e. the problem
  \begin{equation}\label{fvmeth}
 \begin{cases}
      & \displaystyle F(U^*_i)_x = S(U^*_i)H_x, \\[1em]
      & \displaystyle \frac{1}{\Delta x}\int_{I_i} U^*_i(x) \, dx = U_i,
 \end{cases}
 \end{equation}
 has to be solved, where $\Delta x$ is the space step. } From this point of view, the technique introduced here for finite difference methods is easier, since the local problems to be solved are standard Cauchy problems for ODE system \eqref{sblst}. In all the numerical tests considered here, explicit or implicit 
 expressions of the stationary solutions are available, which will allow us to find $U^*_i$ easily. For cases where this is not possible, a numerical method for ODE can be used to compute $U^*_i$ at the stencil of $x_i$ in the spirit of \cite{CastroGomezPares}. 
 
 \rs{However, in spite of their higher complexity, the family of high order finite volumes based on the solutions of problem \eqref{fvmeth} have in general better conservation properties than their finite difference counterparts introduced here. Although all the numerical methods introduced in this article have the property of reducing to a conservative method whenever the source term vanishes, they may fail to be conservative for the conservation laws included in the system, which is not the case for the finite volume methods based on \eqref{fvmeth}. In particular, in the case of the shallow water equations, the methods that preserve all stationary solutions do not preserve total mass; whereas those that preserve only water at rest solutions or a particular stationary solution do. This limitation is related, as it will be seen, to the way in which the right and left reconstructions of the flux are combined in finite difference methods to obtain a stable scheme: due to this, it remains a challenge to design finite difference numerical methods based on flux reconstructions that are essentially non-oscillatory high-order accurate and well-balanced for every stationary solution. In any case, the conservation errors are of the order of the methods and converge to 0 as $\Delta x \to 0$.}

 The case in which $H$ has jump discontinuities will be also considered. At a discontinuity of $H$, a solution $U$ is expected to be discontinuous too and the source term
$S(U)H_x$ cannot be defined within the distributional framework: it becomes a nonconservative product whose meaning has to be specified. There are different mathematical theories that allow one to give a sense to nonconservative products. In the theory developed in \cite{DalMaso95}, nonconservative products are interpreted as  Borel measures whose definition depends on the choice of a family of paths that, in principle, is arbitrary. As the Rankine-Hugoniot conditions, and thus the definition of weak solution, depend on the selected family of paths, its choice has to be consistent with the physics of the problem. Although for general nonconservative systems the adequate selection of paths may be difficult, in the case of systems of balance laws with singular source term there is a natural choice in which the paths are related to the stationary solutions of a regularized system: the interested reader is addressed to \cite{CastroPares2020} for a detailed discussion. \rs{In the particular case of the shallow water system, this choice implies that the admissible stationary weak solutions
still satisfy \eqref{sw-ss}, i.e. the constants $C_i$, $i=1,2$ corresponding to the mass-flux and the total energy cannot change across a discontinuity of $H$.} Although in general finite volume methods can be more easily adapted to deal with nonconservative products than finite difference schemes (see \cite{Pares06}), it will be shown that the numerical methods that preserve every stationary solution can be easily adapted to deal properly with singular source terms.

The organization of the article is as follows: finite difference high order methods based on a reconstruction operator are recalled in Section 2 where the particular example of WENO methods is highlighted. In Section 2 the case in which $H$ is continuous and a.e. differentiable \rf{and the eigenvalues of $J(U)$ do not vanish} is considered. First, the numerical methods that preserve any stationary solution are introduced, together with the proofs of their accuracy and their well-balanced property. Next, methods that preserve \rf{a prescribed set of stationary solutions} are introduced and the conservation property is discussed. The implementation of well-balanced WENO methods is also discussed. 
Section 3 is devoted to the \rs{extension of the method to more complex situations. First, resonant problems are discussed, i.e. situations in which one  of the eigenvalues of $J(U)$ vanishes. Then, a strategy to adapt the technique introduced here to problems whose stationary solutions are unknown or their computation is very costly is discussed. Next, the case in which $H$ is a.e. differentiable and piecewise continuous with isolated jump discontinuities is studied: the definition of the nonconservative product is briefly discussed and the numerical methods introduced in Section 2 are adapted to this case. Finally, the extension to multidimensional  problems is briefly discussed. Section 4 focuses on numerical experiments: the numerical methods are applied to the linear transport equation with linear source term, Burgers' equation with a nonlinear source term, and the shallow water equations. 
Finally, some conclusions are drawn and further developments are discussed.}

\section{Numerical methods}

\subsection{General case}
\rs{Uniform meshes of constant step $\Delta x$ and nodes $\{x_i\}$ will be considered here and the following notation will be also used for the intercells
$$
x_{i + 1/2} = x_i + \frac{\Delta x}{2},\quad \forall i.
$$
The numerical methods to be developed here are based on the high order finite difference conservative schemes  for systems of conservation laws
\begin{equation}\label{scl}
U_t + F(U)_x = 0
\end{equation}
introduced by  Shu and Osher in  \cite{Shu89}. In their approach, the discretization of the derivative of the flux relies on the equality
$$
F(U(x,t))_x = \frac{\widehat{F}(x + \Delta x/2) - \widehat{F}(x-\Delta x /2)}{\Delta x}
$$
that is exactly satisfied if $\widehat{F}(x)$ is a function such that
$$
F(U(x,t)) = \frac{1}{\Delta x}\int_{x-\Delta x/2}^{x + \Delta x/2}\widehat{F}(s) \,ds, \quad \forall x.
$$
Observe that the cell-averages of such a function  would be given by:
$$
\frac{1}{\Delta x}\int_{x_i-\Delta x/2}^{x_i + \Delta x/2}\widehat{F}(s) \,ds = F(U(x_i, t)),\quad \forall i,
$$
and thus a standard reconstruction operator can be used to obtain high-order  approximations  $\widehat{F}_{i+1/2}$ of $\widehat{F}(x_{i+1/2})$
 from the values of the cell-averages of $\widehat{F}$:
$$
\widehat{F}_{i+1/2} = \mathcal{R}(F(U(x_{i-r})), \dots,F(U(x_{i+s}))),
$$
where $\mathcal{S}_i = \{ x_{i-s}, \dots,  x_{i+r} \}$ is the stencil of the reconstruction operator.  ENO or WENO reconstructions are examples of such  operators: see \cite{JiangShu}, \cite{Shu88}, \cite{Shu98}.}

\rs{Once the reconstruction operator has been chosen, the semi-discrete numerical method writes then as follows:
\begin{equation}\label{sdmethscl-1}
\frac{d U_i}{dt} + \frac{1}{\Delta x}   \left( \widehat F_{i + 1/2} - \widehat F_{i-1/2} \right)  = 0,
\end{equation}
where
\begin{equation}\label{sdmethscl-1b}
\widehat{F}_{i+1/2} = \mathcal{R}(F(U_{i-r}), \dots,F(U_{i+s})).
\end{equation}
}  

\rs{A possible extension of these methods for systems of balance laws \eqref{sbl} is given by:
\begin{equation}\label{sdmeth}
\frac{d U_i}{dt} + \frac{1}{\Delta x}   \left( \widehat F_{i + 1/2} - \widehat F_{i-1/2} \right)  = S(U_i) H_x (x_i),
\end{equation}
but in general these methods are not well-balanced}.

\rs{The semidiscrete methods \eqref{sdmethscl-1} and \eqref{sdmeth} can be discretized in time by using TVD-RK methods: see \cite{Gottlieb98}}.

\subsection{WENO reconstructions}

In the particular case of the  WENO reconstruction of order $p = 2k + 1$, two flux reconstructions are computed using the values at the points $x_{i-k}, \dots, x_{i+k}$:
\begin{eqnarray}
\widehat F^L_{i+1/2} &  = &  \rs{\mathcal{R}}^L(F(U_{i-k}), \dots, F(U_{i+k})), \\
\widehat F^R_{i-1/2} & = &  \rs{\mathcal{R}}^R(F(U_{i-k}), \dots, F(U_{i+k})). 
\end{eqnarray}
\rs{$\mathcal{R}^L$ and $\mathcal{R}^R$ represent the so-called left and right-biased reconstructions, i.e. $\mathcal{R}^L$  approximates $F(U(x_{i + 1/2},t))$ on the left
and $\mathcal{R}^R$  approximates $F(U(x_{i - 1/2}, t))$ on the right. The  right-biased reconstructions can be computed with $\mathcal{R}^L$ by reflecting the arguments through  the intercell at which the numerical flux is computed.} 
Once these reconstructions have been computed, an upwind criterion can be chosen to define $ \widehat F_{i+1/2}$. For instance, for scalar problems
\begin{equation}\label{sblsc}
u_t + f(u)_x = s(u) H_x,
\end{equation}
an approximate value $a_{i+1/2}$ of $a(u) = f'(u)$ is chosen and the numerical flux is defined then as follows: 
$$
\widehat f_{i+1/2} = a^+_{i+1/2}  \widehat f^L_{i+1/2} + a^-_{i+1/2}  \widehat f^R_{i+1/2} ,
$$
where $ \widehat f^L_{i+1/2}$, $\widehat f^R_{i+1/2}$,  $\widehat f_{i+1/2}$ represent respectively the left-biased,  the right-biased, and the final WENO reconstruction of the flux at $x_{i+1/2}$,  and
$$a^\pm = \frac{1}{2}(a \pm |a|).
$$
For systems, a matrix $A_{i+1/2}$ with $N$ real different eigenvalues
$$
\lambda_{i+1/2,1}, \dots, \lambda_{i+1/2, N}
$$
that approximates the Jacobian $J(U)$ of the flux function has to be chosen and
then the numerical flux can be defined by
\begin{equation}\label{upwind}
\widehat F_{i+1/2} = P^+_{i+1/2} \widehat F^L_{i+1/2} +   P^-_{i+1/2} \widehat F^R_{i+1/2},
\end{equation}
where
\begin{equation}\label{projectors}
P^\pm_{i+1/2} = K_{i+1/2} D^\pm_{i+1/2}  K^{-1}_{i+1/2}.
\end{equation}
Here $D^\pm_{i+1/2}$ is the diagonal matrix whose coefficients are
$$
 \frac{1}{2}\left(1 \pm \textrm{sign}(\lambda_{i+1/2, j})\right), \quad j = 1, \dots, N,
 $$
 and $K_{i+1/2}$ is a matrix whose columns are eigenvectors.

An alternative approach is to split the flux 
$$
F(U) = F^+(U) + F^-(U)
$$
in such a way that the eigenvalues of the Jacobian $J^+(U)$ (resp. $J^-(U)$) of $F^+(U)$ (resp. $F^-(U)$) are positive (resp. negative). Then, the reconstruction operator is applied to $F^\pm$:
\begin{eqnarray}
\widehat F^{+}_{i+1/2} &  = &  \rs{\mathcal{R}}^L(F^+(U_{i-k}), \dots, F^+(U_{i+k})), \\
\widehat F^{-}_{i-1/2} & = &  \rs{\mathcal{R}}^R(F^-(U_{i-k}), \dots, F^-(U_{i+k})), 
\end{eqnarray}
and finally,
\begin{equation}\label{split}
\widehat F_{i+1/2} =  \widehat F^{+}_{i+1/2} + \widehat F^{-}_{i+1/2}.
\end{equation}
A standard choice is the Lax-Friedrichs flux-splitting:
$$
F^\pm(U) = \frac{1}{2} \left( F(U) \pm \alpha U \right),
$$
where $\alpha$ is the local (WENO-LLF) or global (WENO-LF)  maximum of the absolute value of the eigenvalues of $\{J(U_i)\}$:  see \cite{JiangShu}, \cite{Shu98}.

In both cases (the upwind or the splitting implementations) the values used to compute $\widehat F_{i+1/2}$ are those at the points
$$x_{i-k}, \dots, x_{i+ k +1}$$
so that, in the general notation,  $r = k$ and $s = k+1$.

\section{Well-balanced high-order finite difference methods}\label{ss:wbhofdm}
In order to tackle the difficulties gradually, let us suppose first that $H$ is continuous and a.e. differentiable \rf{and that 
the eigenvalues $J(U)$ are different from 0 for all $U$.}

\subsection{Definition of the method: general case}\label{ss:wbmeth}
\rs{As it was mentioned in the Introduction, the  idea is to write the source term as the derivative of $F(U^*_i(x))$ at $x_i$ using \eqref{mot2},  where $U^*_i$ 
is the solution of the Cauchy problem \eqref{mot1}, and to apply then the reconstruction operator to
the differences $\{ F(U_j) - F(U^*_i(x_j) \}$ to obtain, at the same time, high-order approximation of the flux and the source term.}

The following semi-discrete numerical method  is thus proposed:
\begin{equation}\label{wbsdmeth}
\frac{d U_i}{dt} + \frac{1}{\Delta x}   \left( \widehat{\mathcal{F}}_{i,i + 1/2} - \widehat{\mathcal{F}}_{i,i-1/2} \right)  =  0,
\end{equation}
where the \lq\lq numerical fluxes\rq\rq $\widehat{\mathcal{F}}_{i, i\pm 1/2}$ are computed as follows:
\begin{enumerate}

\item Look for the \rs{solution} $U^*_i (x)$ of the Cauchy problem \eqref{mot1}.

\item Define
    $$ \mathcal{F}_j = F(U_j) - F(U^*_i(x_j)), \quad j = i-1-r,\dots, i+s$$

\item  Compute
\begin{eqnarray*}
\widehat{\mathcal{F}}_{i,i+1/2} & =&  \rs{\mathcal{R}}(\mathcal{F}_{i-r}, \dots, \mathcal{F}_{i+s}), \\
\widehat{\mathcal{F}}_{i,i-1/2} & = & \rs{\mathcal{R}}(\mathcal{F}_{i-1-r}, \dots, \mathcal{F}_{i-1+s}).
\end{eqnarray*}

\end{enumerate}

\begin{remark} 
In the notation $\widehat{\mathcal{F}}_{i,i+1/2}$ the index $i +1/2$ corresponds to the intercell and the index $i$ to the center of the cell where the initial condition of \eqref{Cauchy} is imposed. Therefore, in general
\begin{equation}\label{nocons}
\widehat{\mathcal{F}}_{i,i+1/2} \not= \widehat{\mathcal{F}}_{i+1,i+1/2}
\end{equation}
as one can expect due to the non conservative nature of the system of equations. Notice that two reconstructions have to be computed at every stencil $\mathcal{S}_i$: $\widehat{\mathcal{F}}_{i,i+1/2}$ and $\widehat{\mathcal{F}}_{i+1,i+1/2}$.
\end{remark}

The following result holds:
\begin{proposition} If the numerical method \eqref{sdmeth} is well-defined, \rs{the reconstruction operator $\mathcal{R}$ has order of accuracy $k$, and} the stationary solutions of \eqref{sbl} are smooth, then \rs{the numerical method \eqref{wbsdmeth} has also  order of accuracy $k$}. 
\end{proposition}

\begin{proof}
Let $U(x,t)$ be a smooth solution of \eqref{sbl}. Given a time $t$ and an index $i$,  the reconstruction procedure is applied to 
$\{F(U(x_j, t)) - F(U^{*,t}_i(x_j)) \}_{j=i-r}^{i+s+1}$ to obtain $\widehat{\mathcal{F}}_{i, i\pm1/2}$, where $U^{*,t}_i$
represents the solution of \eqref{sblst} that satisfies
\begin{equation}\label{icCauchyt}
U^{*,t}_i(x_i)=U(x_i, t).
\end{equation}

 One has:
\begin{eqnarray*}
& & \partial_t U(x_i, t) + \frac{1}{\Delta x} \left( \widehat{\mathcal{F}}_{i,i + 1/2} - \widehat{\mathcal{F}}_{i,i-1/2} \right) \\
& & \qquad =  \partial_t U(x_i, t)+ \partial_x F(U)(x_i, t) -\partial_x  F(U_i^{*,t})(x_i) + O(\Delta x^k) \\
& & \qquad =  \partial_t U(x_i, t)+ \partial_x F(U)(x_i, t)   - S(U^{*,t}_i(x_i))\partial_x H (x_i)  + O(\Delta x^k) \\
& & \qquad =  \partial_t U(x_i, t)+ \partial_x F(U)(x_i, t)   - S(U(x_i, t))\partial_x H (x_i)  + O(\Delta x^k) \\
& & \qquad = O(\Delta x^k),
\end{eqnarray*}
where the facts that $U$ is a solution of \eqref{sbl} and $U^{*,t}_i$  a stationary solution satisfying \eqref{icCauchyt} have been used. 
\end{proof}

Observe that the method is well-defined if the first step of the reconstruction procedure can be always performed, i.e. if  for every $i$
the Cauchy problem
\begin{equation}\label{Cauchy}
\left\{
\begin{array}{l}
\displaystyle \frac{d\ }{dx} F(U) = S(U)H_x,  \smallskip\\
U(x_i) = U_i,
\end{array}
\right.
\end{equation}
 has a unique solution whose interval of definition contains the extended stencil  $\widehat{\mathcal{S}}_i = \{ x_{i-r-1}, \dots, x_{i+s}\}$. 
 
\rf{Since the eigenvalues of $J(U)$ are assumed  to be different from 0, \eqref{Cauchy} is equivalent to
  \begin{equation}\label{Cauchy2}
\left\{
\begin{array}{l}
\displaystyle \frac{d U }{dx} = J(U)^{-1} S(U)H_x,  \smallskip\\
U(x_i) = U_i,
\end{array}
\right.
\end{equation}
and,  under the adequate smoothness assumptions, this Cauchy problem has  a unique maximal solution  $U^*_i$ defined in an interval $(\alpha, \beta)$.
In this case, only two things can happen:
\begin{itemize}
    \item If $\widehat{\mathcal{S}}_i \subset (\alpha, \beta)$ then \eqref{wbsdmeth} can be used to update $U_i$ provided that $U^*_i$  can be computed.
    \item If $x_{i-r-1} < \alpha$ or $\beta < x_{i+s}$  then $U_j$, $j = i-r-1, \dots, i+s$ cannot be the values of a stationary solution $U^*$ at the points of the stencil (otherwise, $U^*$ would be a solution of \eqref{Cauchy2} defined in an interval bigger than $(\alpha, \beta)$). Therefore, in this case there is no need of a well-balanced method and \eqref{sdmeth} can be used to update $U_i$.
\end{itemize}
}

\subsection{Well-balanced property}

\rs{Let us suppose that the reconstruction operator satisfies
$$
 0 = \mathcal{R}(0, \dots, 0).
 $$
Then, the numerical method \eqref{wbsdmeth}  is well-balanced in the sense given by the following
\begin{proposition}\label{prop:wb}
Given a stationary solution $U^*$ of \eqref{sbl}, the vector of its point values $\{ U^*(x_i) \}$ is an equilibrium of the ODE system given by the semi-discrete method \eqref{wbsdmeth}.
\end{proposition}
}
\begin{proof}
\rs{Observe that, when the algorithm proposed in Section \ref{ss:wbmeth} is  applied to $\{U^*(x_i)\}$ to compute the numerical fluxes $\mathcal{F}_{i,i\pm 1/2}$, at the first stage one has
$$
U^{*}_i \equiv U^*, 
$$
since $U^*$ solves \eqref{Cauchy2}, 
and thus
$$
F(U^*(x_j)) - F(U^{*}_i(x_j)) = 0, \quad  j= i-1 -r,\dots, i+s.
$$
Therefore,
\[
\widehat{\mathcal{F}}_{i,i + 1/2} - \widehat{\mathcal{F}}_{i,i-1/2} = 0,
\]
as we wanted to prove.}

\end{proof}

\subsection{Numerical method with WENO reconstructions}
Let us discuss the implementation of the numerical method \eqref{wbsdmeth} in the particular case of WENO reconstructions with the upwind or the flux-splitting approach.

\subsubsection{Upwind approach}

The implementation in this case is as follows: once the solution $U^*_i$ has been computed:
\begin{itemize}
    \item Define
    $$ \mathcal{F}_j = F(U_j) - F(U^*_i(x_j)), \quad j = i-k-1,\dots, i+k+1$$    
    
    \item  Compute
    \begin{eqnarray*}
    \widehat{\mathcal{F}}^L_{i,i+1/2} & = & \rs{\mathcal{R}}^L(\mathcal{F}_{i-k}, \dots, \mathcal{F}_{i+k}), \\
    \widehat{\mathcal{F}}^R_{i,i-1/2}  & = & \rs{\mathcal{R}}^R(\mathcal{F}_{i-k}, \dots, \mathcal{F}_{i+k}), \\
    \widehat{\mathcal{F}}^R_{i,i+1/2} & = & \rs{\mathcal{R}}^R(\mathcal{F}_{i-k+1}, \dots, \mathcal{F}_{i+k+1}),\\
    \widehat{\mathcal{F}}^L_{i,i-1/2} & = & \rs{\mathcal{R}}^L(\mathcal{F}_{i-k-1}, \dots, \mathcal{F}_{i+k-1}).
    \end{eqnarray*}

\item Choose intermediate matrices $A_{i\pm 1/2}$.

\item Define
\begin{eqnarray*}
\widehat{\mathcal{F}}_{i, i+1/2} & = &  P^+_{i+1/2} \widehat{\mathcal{F}}^L_{i, i+1/2} + P^-_{i+1/2} \widehat{\mathcal{F}}^R_{i,i+1/2},\\
\widehat{\mathcal{F}}_{i, i-1/2} & = & P^+_{i-1/2}  \widehat{\mathcal{F}}^L_{i, i-1/2} + P^-_{i-1/2} \widehat{\mathcal{F}}^R_{i,i-1/2},
\end{eqnarray*}
where the projection matrices $P^\pm_{i \pm 1/2}$ are given by \eqref{projectors}.

\end{itemize}

\subsubsection{Flux-splitting approach}
The implementation of WENO with splitting approach will be as follows: once the solution $U^*_i$ has been computed:
\begin{itemize}
    \item Define
    \begin{eqnarray*}
    \mathcal{F}^+_j = F^+(U_j) - F^+(U^*_i(x_j)), \quad j = i-k-1,\dots, i+k\\
    \mathcal{F}^-_j = F^-(U_j) - F^-(U^*_i(x_j)), \quad j = i-k,\dots, i+k+1
    \end{eqnarray*}
    
    \item  Compute
    \begin{eqnarray*}
    \widehat{\mathcal{F}}^+_{i,i+1/2} & = & \rs{\mathcal{R}}^L(\mathcal{F}^+_{i-k}, \dots, \mathcal{F}^+_{i+k}), \\
    \widehat{\mathcal{F}}^-_{i,i-1/2}  & = & \rs{\mathcal{R}}^R(\mathcal{F}^-_{i-k}, \dots, \mathcal{F}^-_{i+k}), \\
    \widehat{\mathcal{F}}^-_{i,i+1/2} & = & \rs{\mathcal{R}}^R(\mathcal{F}^-_{i-k+1}, \dots, \mathcal{F}^-_{i+k+1}),\\
    \widehat{\mathcal{F}}^+_{i,i-1/2} & = & \rs{\mathcal{R}}^L(\mathcal{F}^+_{i-k-1}, \dots, \mathcal{F}^+_{i+k-1}).
    \end{eqnarray*}

\item Define
$$
\widehat{\mathcal{F}}_{i, i\pm 1/2}  =   \widehat{\mathcal{F}}^-_{i, i\pm 1/2} + \widehat{\mathcal{F}}^+_{i,i\pm1/2}.
$$

\end{itemize}

In the particular case of the Lax-Friedrichs splitting, the reconstruction operators $\mathcal{F}^L$ and
$\mathcal{F}^R$ will be applied to the values
$$
F(U_j) - F(U^*_i(x_j)) \pm \alpha (U_j - U^*_i(x_j))
$$
in the corresponding stencil. Again $\alpha$ is the local or global maximum eigenvalue of $\{J(U_i)\}$: although the numerical viscosity \rf{can be small when $U_j$  is close to $U_i^*(x_j)$}, the numerical method has been shown to be stable under a CFL number of 1/2 in all the test cases considered in Section \ref{ss:tests}.

\begin{remark}
Observe that, while for a conservative system 2 reconstructions are computed at every stencil $x_{i-k}, \dots, x_{i+k}$, 4 reconstructions have to be computed now: 2 using $U^*_i$, 1 using $U^*_{i-1}$,  1 using $U^*_{i+1}$.
\end{remark}

\subsection{Numerical methods that preserve \rf{a family of stationary solutions}}\label{ss:wb1}

\rf{The strategy described in Section \ref{ss:wbhofdm}  can be easily adapted to obtain schemes that only preserve a prescribed set of stationary  solutions: this would be the case if, for instance, one is interested in the design of numerical methods for the shallow water model that only preserve the water-at-rest solutions \eqref{sw-war}.  If, as in this example, the set to be preserved is a $k$-parameter family of stationary solutions,
$$
U^*(x; C_1,\dots, C_k),
$$
with $k < N$, where $N$ is the number of  unknowns, the following numerical method is proposed:}

\rf{
\begin{equation}\label{wbfam}
\frac{d U_i}{dt} + \frac{1}{\Delta x}   \left( \widehat{\mathcal{F}}_{i,i + 1/2} - \widehat{\mathcal{F}}_{i,i-1/2} \right)  =  (S(U_i)- S(U_i^*(x_i)))H_x(x_i),
\end{equation}
where  $\widehat{\mathcal{F}}_{i,i\pm 1/2}$ and $U^*_i$ are computed as follows:}

\begin{enumerate}

\item \rf{ Find $C^i_1, \dots, C^i_k$ such that:
\begin{equation}\label{firststagek}
 u_{j_l}^* (x_i; C^i_1, \dots, C^i_k)  =  u_{i, j_l}, \quad l=1, \dots, k,
\end{equation}
where  $u^*_{j}$, $ u_{i,j}$ denote respectively the $j$th component of $U^*$ and $U_i$ and $\{ j_1, \dots, j_k\}$ is a set of $k$ indices that is
predetermined in order to have the same number of unknowns and equations in \eqref{firststagek}.
Then, define:
$$
U^*_i(x) = U^*(x;C^i_1, \dots, C^i_k).
$$
}

\item \rf{Define
    $$ \mathcal{F}_j = F(U_j) - F(U^*_i(x_j)), \quad j = i-1-r,\dots, i+s$$    
}

\item  \rf{Compute
\begin{eqnarray*}
\widehat{\mathcal{F}}_{i,i+1/2} & =&  {\mathcal{R}}(\mathcal{F}_{i-r}, \dots, \mathcal{F}_{i+s}), \\
\widehat{\mathcal{F}}_{i,i-1/2} & = & {\mathcal{R}}(\mathcal{F}_{i-1-r}, \dots, \mathcal{F}_{i-1+s}).
\end{eqnarray*}
}

\end{enumerate}

Observe that, if $k=N$, \eqref{firststagek} is equivalent to solving the Cauchy problem.

\rf{It can be easily shown that, if the numerical method is well-defined (i.e. if the equation \eqref{firststagek} has a unique solution for every $i$) and the 
stationary solutions of the family are smooth, then the numerical method has the order of accuracy of the reconstruction operator and it preserves all the stationary solutions of the family.}

\rf{Let us apply this methodology to derive a family of numerical methods that preserve the water-at-rest solutions of the shallow water model. In this  case, the family of stationary solutions to be preserved is given by:
\begin{equation}
    h^*(x; \eta^*) = \eta^* + H(x), \quad q^*(x) = 0,
\end{equation}
where $\eta^*$ is an arbitrary constant corresponding to the elevation of the undisturbed water surface. Given a state $U_i = [h_i, q_i]^T$, in the first stage of
the algorithm that computes the numerical fluxes, we select the solution of this family that satisfies:
\[
h^*(x_i, \eta^*) = h_i,
\]
i.e. the first index is selected to fix the constant. We therefore aim to preserve the stationary solution with flat water surface at height $\eta^* = H(x_i) + \eta_i$ everywhere. 
The selected stationary solution of the family is thus
\[
U_i^*(x) = \left[ \begin{array}{c} \eta_i + H(x) \\ 0 \end{array} \right],
\]
so that the numerical fluxes $\widehat{F}_{i,i\pm 1/2}$ are computed by applying the reconstruction operator to:
\[
\mathcal{F}_j = F(U_j) - F(U_i^*(x_j)) = \left[ \begin{array}{c}  q_j \\ \displaystyle \frac{q_j^2}{h} + \frac{g}{2}h_j^2  -  \frac{g}{2} h^*_i(x_j) \end{array}\right]
= \left[ \begin{array}{c}  q_j \\ \displaystyle \frac{q_j^2}{h_j} + \frac{g}{2}(\eta_j^2  - \eta_i^2) + g(\eta_j - \eta_i)H(x_j) \end{array}\right],
\]
where $\eta_j = h_j - H(x_j)$. On the other hand:
\[
 S(U_i)- S(U_i^*(x_i)) = 0,
 \]
 so that the numerical method writes as follows:
\begin{equation}\label{wbswar}
\frac{d U_i}{dt} + \frac{1}{\Delta x}   \left( \widehat{\mathcal{F}}_{i,i + 1/2} - \widehat{\mathcal{F}}_{i,i-1/2} \right)  =  0.
\end{equation}
Notice that the numerical method obtained can be interpreted as a discretization of the following equivalent formulation of the shallow water system:
\[
\left\{
\begin{array}{l}
\displaystyle  \eta_t + q_x = 0, \\
\displaystyle q_t + \left( \frac{q^2}{\eta + H} + \frac{g}{2}\eta^2 + g\eta H \right)_x  =  g \eta H_x.
\end{array}
\right.
\]
}

\rf{Coming back to the general case, let us remark that, in the case in which the family of stationary solutions has only one element,  i.e. if there is only one stationary solution $U^*(x)$ to preserve,  the previous algorithm can be easily adapted: \eqref{firststagek} is skipped and $U^*_i \equiv U^*$ is selected in the first step. In this case the expression of the numerical method is as follows:
\begin{equation}\label{wb1s}
\frac{d U_i}{dt} + \frac{1}{\Delta x}   \left( \widehat{\mathcal{F}}_{i + 1/2} - \widehat{\mathcal{F}}_{i-1/2} \right)  =  (S(U_i)- S(U^*(x_i)))H_x(x_i).
\end{equation}
Note that unlike in eq. \eqref{wbfam}, here we have $\widehat{\mathcal{F}}_{i + 1/2}$ instead of $\widehat{\mathcal{F}}_{i,i + 1/2}$ since the steady state solution to use in the reconstruction does not vary with cell $i$. More precisely, $\widehat{\mathcal{F}}_{i\pm 1/2}$ and $U^*_i$ are computed using the algorithm:}

\begin{enumerate}

\item \rf{Define
    $$ \mathcal{F}_j = F(U_j) - F(U^*(x_j)), \quad j = i-1-r,\dots, i+s.$$    
}

\item  \rf{Compute
\begin{eqnarray*}
\widehat{\mathcal{F}}_{i+1/2} & =&  {\mathcal{R}}(\mathcal{F}_{i-r}, \dots, \mathcal{F}_{i+s}), \\
\widehat{\mathcal{F}}_{i-1/2} & = & {\mathcal{R}}(\mathcal{F}_{i-1-r}, \dots, \mathcal{F}_{i-1+s}).
\end{eqnarray*}
}

\end{enumerate}

\rf{Notice that, in this particular case, the numerical fluxes only depend on $i+1/2$. Furthermore, observe that the numerical method can be also derived as follows: subtract from \eqref{sbl} the equation satisfied by the stationary solution $U^*$ to obtain:
 \[  U_t+(F(U)-F(U^*))_x= (S(U)-S(U^*)) H_x \]
 and apply a reconstruction operator to compute the derivative of the 'flux' function. This strategy  is well-known and has been applied in many  fields, like in atmospheric sciences, where $U^*$ represents the background gravitational effect.  Although it is not new, we mention this example to put it in context in the more general framework introduced here.}

 \rf{Before finishing this discussion, let us mention another case in which a numerical method that preserves a family of stationary solutions is useful. Let us consider now the Euler equations of gas dynamics  with  source term for the simulation of the flow of a gas in a linear gravitational field:
\begin{equation}\label{Euler}
\left\{
\begin{array}{l}
\rho_t + (\rho u)_x = 0, \\[0.2cm]
(\rho u)_t + (\rho u^2 + p)_x=  - g\rho , \\[0.2cm]
(E)_t + (u(E+p))_x =  -g \rho u.
\end{array}
\right.
\end{equation}
Here, $\rho \geq 0$ is the density, $u$ the velocity,  $ m=\rho u$ the momentum, $p \geq 0$ the pressure,  $E$ the total energy per unit volume, and $H(x)$ the gravitational potential. Furthermore, the internal energy $e$ is given by $\rho e=E- \frac{1}{2} \rho u^2$. Pressure is determined from $e$ through the equation of state (EOS). Here we suppose for simplicity an ideal gas, therefore
\[
p=(\gamma -1)\rho e,
\]
where $\gamma>1$ is the adiabatic constant.  } 

\rf{System \eqref{Euler} can be written in the form \eqref{sbl} with $H(x) = - gx$,
\[
U=\left[ \rho, \, \rho u, \, E\right]^T, \quad F(U)=\left[ \rho u, \, \rho u^2 + p, \, u (E+p)\right]^T, \quad S(U)=\left[0,\, -\rho, -\rho u\right]^T.
\]
The  hydrostatic equilibrium solutions Euler equations with gravity satisfy
\[
u(x,t)=0   \quad p_x=-\rho H_x.
\]
A family of isothermal stationary solutions depending on two positive parameters $C_1$, $C_2$ is given by
\begin{equation}\label{termalH}
\rho^*(x)=C_1 e^{-gx} , \quad p^*(x)=C_1 e^{-gx} + C_2, \quad u^*(x) = 0, \quad E^*(x) = \frac{p^*(x)}{\gamma -1}.
\end{equation}
}

\rf{High-order numerical methods that preserve this family of stationary solutions can be derived following the strategy proposed here: given the state $U_i = [\rho_i, \rho_i u_i, E_i]^ T$, first we look for
the stationary solution of the family such that
\[ \rho^*(x_i) = \rho_i, \quad p^*(x_i) = p_i,\]
which is
\begin{equation}\label{termalHi}
\rho_i^*(x)= \rho_i e^{-g(x - x_i)} , \quad p_i^*(x)= p_i - \rho_i + \rho_i e^{-g(x - x_i)}, \quad u_i^*(x) = 0, \quad E_i^*(x) = \frac{p_i^*(x)}{\gamma -1}.
\end{equation}
In this case,  one has: 
\[
S(U_i) - S(U^*(x_i)) = \left[ \begin{array}{c} 0 \\ 0 \\ -\rho_i u_i \end{array} \right].
\]
Taking into account the expression of $p^*_i$, it can be easily checked that the numerical method \eqref{wbfam} can be equivalently written  in flux-source term form as follows:
\begin{equation}\label{sdmethEuler1}
\frac{d U_i}{dt} + \frac{1}{\Delta x}   \left( \widehat{F}_{i + 1/2} - \widehat{F}_{i-1/2} \right)  =  
\left[ \begin{array}{c} 0 \\ \displaystyle \rho_i e^{gx_i}\frac{\psi_{i+1/2} - \psi_{i-1/2}}{\Delta x} \\ -g\rho_i u_i \end{array} \right].
\end{equation}
where $\widehat{F}_{i \pm 1/2},~ \psi_{i \pm 1/2}$ represent the WENO reconstructions of the flux function $F(U_j)$ 
and $e^{-gx_j}$ respectively. Please observe that this is just a rewrite of eq. \eqref{wbfam}. That is, the reconstruction procedure has been applied to $(F(U_j) - F(U^*_i(x_j)))$, and has only been separated into the terms in the left and right hand side to make more explicit the comparison to \cite{Xing13} (\eqref{sdmethEuler2} below). Note in particular that the choice of WENO weights in the reconstruction of $\exp(-gx_j)$ and $F(U_j)$ is not independent from each other; both use the same set of weights, arising from applying the reconstruction procedure to $(F(U_j) - F(U^*_i(x_j)))$}

\rf{Let us compare this numerical method with the one proposed in \cite{Xing13}. The strategy developed in this reference relies on the equivalent formulation of the system
\begin{equation}\label{Euler2}
\left\{
\begin{array}{l}
\rho_t + (\rho u)_x = 0, \\[0.2cm]
(\rho u)_t + (\rho u^2 + p)_x=  \rho \exp(gx)(\exp(-gx))_x, \\[0.2cm]
(E)_t + (u(E+p))_x =  \rho u \exp(gx)(\exp(-gx))_x.
\end{array}
\right.
\end{equation}
The numerical method writes as follows:
\begin{equation}\label{sdmethEuler2}
\frac{d U_i}{dt} + \frac{1}{\Delta x}   \left( \widehat{F}_{i + 1/2} - \widehat{F}_{i-1/2} \right)  =  
\left[ \begin{array}{c} 0 \\ \rho_i e^{gx_i} \\ \rho_i u_i e^{gx_i}\end{array} \right] \displaystyle 
\frac{\psi_{i+1/2} - \psi_{i-1/2}}{\Delta x},
\end{equation}
where $\widehat{F}_{i \pm 1/2}, \psi_{i \pm 1/2}$ represent again the WENO reconstructions of the flux function $F(U_j)$ 
and $ \exp(-gx_j)$ that do not coincide with the ones  appearing in \eqref{sdmethEuler1}: although the WENO coefficients used to compute both reconstructions are again the same, in this case, they rely on the smoothness indicators corresponding to the fluxes $F(U_j)$. The numerical methods are  closely related, the only differences being:}
\begin{itemize}
\item \rf{No explicit reformulation of the source term has been done to obtain \eqref{sdmethEuler1}.}
\item \rf{The  equality $g = \exp(gx)(\exp(-gx))_x$ is not used in the third equation of \eqref{sdmethEuler1}.}
\item \rf{While WENO coefficients in \eqref{sdmethEuler2} take into account the smoothness of the numerical solution, those in \eqref{sdmethEuler2} take into account the smoothness of the fluctuations with respect to the local equilibrium. If the stationary solutions are assumed to be smooth both smoothness indicators should be similar. Nevertheless, if the stationary solutions are discontinuous like in Section \ref{s:Hdisc}, the difference may be important.}
\end{itemize}

\subsection{Conservation property}\label{ss:mc}
\rf{We show here that the methods introduced in Sections \ref{ss:wbmeth} and \ref{ss:wb1} reduce to conservative schemes when $H$ is locally constant  provided that  the only stationary solutions of the homogeneous problem \eqref{scl} are  constant (which is the case for the shallow water system) and that the reconstruction operator satisfies the following property: given $\mathcal{F}_{i-r}, \dots, \mathcal{F}_{i+s}$ in $\mathbb{R}^N$ and an arbitrary vector $V$ in $\mathbb{R}^N$, the following equality holds
\begin{equation}\label{eq:consproprec}
{\mathcal{R}}(\mathcal{F}_{i-r} - V, \dots, \mathcal{F}_{i+s}- V) = {\mathcal{R}}(\mathcal{F}_{i-r}, \dots, \mathcal{F}_{i+s}))  - V.
\end{equation}
}

\rf{In effect, let us assume that \eqref{eq:consproprec} is satisfied. Then one has:
 $$ \mathcal{F}_j = F(U_j) - F(U^*_i(x_j)), \quad j = i-r-1,\dots, i+s$$    
 where $U^*_i$ is a stationary solution defined in the interval $[x_{i-r-1}, x_{i+s}]$. If $H$ is constant in this interval, then $U^*_i$ is also constant. Therefore,
 due to \eqref{eq:consproprec} one has
\begin{eqnarray*}
\widehat{\mathcal{F}}_{i,i+1/2} & =&  \widehat{F}_{i+1/2} - F(U^*_i), \\
\widehat{\mathcal{F}}_{i,i-1/2} & =& \widehat{F}_{i-1/2} - F(U^*_i),
\end{eqnarray*}
where
\begin{eqnarray*}
F_{i+1/2} & =&  {\mathcal{R}}(F(U_{i-r}), \dots, F(U_{i+s})), \\
F_{i-1/2} & = & {\mathcal{R}}(F(U_{i-r-1}), \dots, F(U_{i+s-1}),
\end{eqnarray*}
and thus the numerical method reduces to \eqref{sdmethscl-1}. Notice  that both the upwind and the splitting versions of the WENO reconstructions satisfy \eqref{eq:consproprec}.}

\rf{Nevertheless, unlike the high order finite volume methods based on a similar principle discussed in \cite{CastroPares2020}, the methods introduced here are not conservative in general for the conservative subsystems of \eqref{sbl}  when $H$ is not locally constant. In  effect, let us assume that there exists $I \in \{1, \dots, N\}$ such that:
$$
S(U) = \left[ 0, \dots ,0, S_{I+1}(U),  \dots , S_N(U) \right]^T, 
$$
i.e. the first $I$ equations of \eqref{sbl} are conservation laws: this is the case for the shallow water system with $I = 1$.
Observe that, since $U^*_i$ is a stationary solution, one has
\[
\partial_x f_l(U^*_{i}(x)) = 0, \quad l=1, \dots, I,
\]
so that the first $I$ components   of the function $F(U^*_i(x))$ are constant. Therefore, if one had the equality
\begin{equation}\label{eq:consproprec2}
{\mathcal{R}_l}\bigl(F(U_{i-r}) - F(U_i^*(x_{i-r})) , \dots, F(U_{i+s})- F(U_i^*(x_{i+s}))\bigr) = {\mathcal{R}}_l\bigl(F(U_{i-r}) , \dots, F(U_{i+s}) \bigr) - f_l(U^*_i), 
\end{equation}
for $l= 1, \dots,I$, where ${\mathcal{R}_l}$ represents the $l$th component of the reconstruction, then the numerical method would reduce to a conservative one for the first $I$ equations, as it can be easily checked.  Since the first $I$ components of $F(U_i^*)$ are constant,   property \eqref{eq:consproprec2} seems to be similar to \eqref{eq:consproprec} and one can expect  that  WENO reconstructions also satisfy it, but this is not true in general. Let us see why: in the case of the upwind implementation, one has
\[
\begin{split}
& \widehat{\mathcal{F}}_{i,i+1/2} = P^+_{i+1/2}  {\mathcal{R}^L}\bigl(F(U_{i-k}) - F(U_i^*(x_{i-k})) , \dots, F(U_{i+k})- F(U_i^*(x_{i+k}))\bigr) \\
& \qquad + P^-_{i+1/2}  {\mathcal{R}^R}\bigl(F(U_{i-k+1}) - F(U_i^*(x_{i-k+1})) , \dots, F(U_{i+k+1})- F(U_i^*(x_{i+k+1}))\bigr) 
\end{split}
\]
If all the eigenvalues are positive, then $P_{i+1/2}^+ = I$, $P_{i+1/2}^- = 0$ and thus
\[
\widehat{\mathcal{F}}_{i,i+1/2} =  {\mathcal{R}^L}\bigl(F(U_{i-k}) - F(U_i^*(x_{i-k})) , \dots, F(U_{i+k})- F(U_i^*(x_{i+k}))\bigr).
\]
Therefore, since the reconstructions are computed component by  component one has}
\rf{
\[
\widehat{f}_{i, i+1/2; l} =  {\mathcal{R}_l^L}\bigl(F(U_{i-k}) , \dots, F(U_{i+k})\bigr) -  f_l(U_i^*),\quad l=1, \dots, I, 
\]
 where $\widehat{f}_{i, i+1/2; l}$ represents the $l$th component of $\widehat{\mathcal{F}}_{i,i+1/2}$. Analogously,  one has
 \[
\widehat{f}_{i, i-1/2; l} =  {\mathcal{R}_l^L}\bigl(F(U_{i-k-1}) , \dots, F(U_{i+k-1})\bigr) -  f_l(U_i^*),\quad l=1, \dots, I, 
\]
and then the numerical method is conservative for the first $I$ equations. And the same happens if all the eigenvalues are negative. Nevertheless, in the general case
the product by the projection matrices $P^\pm_{i+1/2}$ mixes the different variables and \eqref{eq:consproprec2} is not satisfied in general.}

\rf{In the case of the splitting implementation, WENO reconstructions are computed as follows: 
\[
\begin{split}
& \widehat{\mathcal{F}}_{i,i+1/2}  = {\mathcal{R}^L}\bigl(F^+(U_{i-k}) - F^+(U_i^*(x_{i-k})) , \dots, F^+(U_{i+k})- F^+(U_i^*(x_{i+k}))\bigr) \\
& \qquad + {\mathcal{R}^R}\bigl(F^-(U_{i-k+1}) - F^-(U_i^*(x_{i-k+1})) , \dots, F^-(U_{i+k+1})- F^-(U_i^*(x_{i+k+1}))\bigr) ,
\end{split}
\]
but now the first $I$ components of $F^\pm(U^*_i(x))$
\[
f_l^\pm(U^*_i(x)) = \frac{1}{2} \left(f_l(U^*_i) \pm \alpha u^*_l(x) \right), \quad l = 1, \dots, I
\]
are not constant in general: while $f_l(U^*_i)$ is constant, this is not the case in general for $u_l^*$, so that  \eqref{eq:consproprec2} is not satisfied in general. 
}

\rf{
Nevertheless, there are some exceptions in which conservation can be proved for the first $I$ components  of the system: this is the case, for instance, for the methods that only preserve one stationary solution. In effect,  in this case the first $I$ equations of \eqref{wb1s} write as follows:
$$
\frac{d u_{l,i}}{dt} + \frac{1}{\Delta x}   \left( \widehat{{f}}_{l, i + 1/2} - \widehat{{f}}_{l, i-1/2} \right)  =  0,\quad l = 1, \dots, I,
$$
where $u_{j,i}$ and $\widehat{{f}}_{j, i + 1/2}$ represent respectively the $j$th component of $U_i$ and $\widehat{\mathcal{F}}_{i + 1/2}$. It is thus conservative for the first $I$ equations.
}

\rf{Another exception is the family of numerical methods that preserve the water-at-rest solutions introduced  in the previous section: in this case, WENO  reconstructions with the global Lax-Friedrichs splitting strategy lead to schemes for which the mass is conserved. Observe that, in this case, the first component
of 
\[ F^\pm(U_{i+j}) - F^\pm(U^*_i(x_{i+j}))
\]
writes as follows
\[
\frac{1}{2}\left(q_{i+j} \pm \alpha (h_{i+j} - \eta_i - H(x_{i+j})\right) = \frac{1}{2}\left(q_{i+j} \pm \alpha \eta_{i+j} \right) \mp \frac{\alpha}{2} \eta_i,
\]
and thus, using \eqref{eq:consproprec} one has:}
\rf{
\[
 \widehat{f}_{i, i+1/2; 1} = \frac{1}{2}\left( \hat f^-_{i+1/2} +  \hat f^+_{i+1/2} \right) ,
\]
where }
\rf{
\[ 
\begin{split}
\hat f^-_{i+1/2} & =  \mathcal{R}_1^R\left( q_{i-k+1} - \alpha \eta_{i-k+1}, \dots,  q_{i+k+1}- \alpha \eta_{i+k+1} \right),\\
\hat f^+_{i+1/2}  & =  \mathcal{R}_1^L\left(  q_{i-k} + \alpha \eta_{i-k} , \dots,  q_{i+k} + \alpha \eta_{i+k}  \right). 
\end{split}
\]
The numerical mass fluxes do not depend on the stationary solution $U^*_i$  and thus mass is conserved. 
}

\rf{Let us finally mention that, in some cases, conservation can be restored for the first $I$ equations using an adequate splitting. For instance, 
the numerical method \eqref{sdmethEuler2} has in principle the same  difficulty concerning mass conservation. In \cite{Xing13}, the standard Lax-Friedrichs splitting is replaced by
\[
F^\pm(U_j) = 
\frac{1}{2} \left( \left[ \begin{array}{c} \rho_j u_j \\ \rho_j u_j^2 + p_j \\ (E_j + p_j) u_j \end{array}\right]
\pm \alpha'\left[ \begin{array}{c} \rho_j e^{gx_j} \\ \rho_j u_j e^{gx_j} \\ E_je^{gx_j}\end{array}\right]\right)
\]
with an adequate choice of $\alpha'$: the corresponding numerical method is well-balanced (since the viscous term is constant for a stationary solution) and mass-conservative (since the expression of the viscous term does not depend on $i$). A similar treatment could be applied to \eqref{sdmethEuler1},
by defining:
\[
\mathcal{F}^\pm_j =  \left[ \begin{array}{c} \rho_j u_j \\ \rho_j u_j^2 + p_j - p^*_i(x_j) \\ (E_j + p_j) u_j \end{array}\right] \pm \alpha' \left[ \begin{array}{c} \rho_j e^{gx_j} \\ \rho_ju_j e^{gx_j} \\ E_je^{gx_j}\end{array}\right], \quad j = i-k-1,\dots, i+k.
\]
A similar procedure may be followed for the numerical methods that preserve every stationary solution of the shallow water equations: if we define
\[
\mathcal{F}^\pm_j =
 \left[ \begin{array}{c} q_j \\ 
 \displaystyle \frac{q_j^2}{h_j} + \frac{g}{2}h_j^2 - \frac{q_i^2}{h^*(x_j)} - \frac{g}{2}h^*_i(x_j)^2  \end{array}\right] \pm \alpha \left[ \begin{array}{c} \displaystyle h_j - H(x_j) + \frac{1}{2g} \frac{q_j^2}{h_j^2} \\ q_j \end{array}\right], \quad j = i-k-1,\dots, i+k,
\]
then the numerical method is still well balanced, since the viscous term is constant for stationary solutions, and the mass is preserved since again the viscous term does not depend on the interval $I_i$. Nevertheless, if we write the viscous term in $\eta, u$ variables
$$
G(\eta,q) = \left[  \begin{array}{c} \displaystyle \eta-  \frac{1}{2g} \frac{q^2}{(\eta + H)^2}  \\ q \end{array} \right],
$$
and compute its gradient
$$
\nabla G(\eta,q) = \left[ \begin{array}{cc} \displaystyle 1 - \frac{u^2}{gh} & \displaystyle \frac{u}{gh}  \\  0 & 1 \end{array} \right],
$$
we can  see that their eigenvalues are positive only if  the flow is subcritical: in this case, the eigenvalues of the Jacobian 
of $\mathcal{F}^\pm_j$ may be made positive/negative by taking $\alpha$ large enough, but for supercritical or transcritical flows, this approach is not expected to give stable numerical methods. Moreover, the eigenvectors of $\nabla G$ are different from those of the Jacobian of $F(U)$ so that not even for subcritical flows the stability is guaranteed, as it will be seen in  Section \ref{ss:tests}.
}

\section{Extensions of the methods}

\subsection{\rs{Unknown  stationary solutions}}\label{ss:Cauchynum}
\rs{In Section \ref{ss:wb1} it has been assumed that the solutions of the Cauchy problem \eqref{Cauchy2} were known or easy to compute. If it is not the case,
the solutions of these Cauchy problems can be approximated using an ODE solver, for instance a one-step method, whose order of accuracy is higher than that of the reconstruction operator, as it is done in \cite{CastroGomezPares} for finite volume methods.
The steps to compute the numerical fluxes are in this case as follows:} 

\begin{itemize}
    \item \rs{Compute approximations $\widetilde U^*_j$, $j= i+1, \dots, i+s$ of the solution of the Cauchy problem \eqref{Cauchy2} at the points 
    $x_{i+1}, \dots, x_{i+s}$ by applying an ODE solver
    \begin{equation}\label{forward}
    U^*_{k+1} = U^*_k + h \Phi_h(z_k, U^*_k), \quad k =0, \dots, sK-1,
    \end{equation}
    in a mesh of step $h = \Delta x_k/K$,  where $K$ is a positive integer. This new mesh is designed so  that 
    \[
    x_{i+j} = z_{jK}, \quad j=0, \dots, s,
    \]
    and then $\widetilde U^*_j = U^*_{jK}$.
    }
    
   \item \rs{Compute approximations $\widetilde U^*_j$, $j= i-r-1, \dots, i-1$ of the solution of the Cauchy problem \eqref{Cauchy2} at the points 
    $x_{i-r-1}, \dots, x_{i-1}$ by a backward application of the ODE solver
    \begin{equation}\label{backward}
    U^*_{-k-1} = U^*_{-k} - h \Phi_{-h}(z_{-k}, U^*_{-k}), \quad k =0, \dots, (r+1)K-1,
    \end{equation}
    in a mesh of step $h = \Delta x_k/K$,  where $K$ is a positive integer. Again 
    \[
    x_{i-j} = z_{-jK}, \quad j=0, \dots, r+1,
    \] 
    so that $\widetilde U^*_{i-j} = U^*_{-jK}$.
    }

\item \rs{Define
    $$ \mathcal{F}_j = F(U_j) - F(\widetilde U^*_j), \quad j = i-1-r,\dots, i+s.$$    }

\item  \rs{Compute
\begin{eqnarray*}
\widehat{\mathcal{F}}_{i,i+1/2} & =&  \rs{\mathcal{R}}(\mathcal{F}_{i-r}, \dots, \mathcal{F}_{i+s}), \\
\widehat{\mathcal{F}}_{i,i-1/2} & = & \rs{\mathcal{R}}(\mathcal{F}_{i-1-r}, \dots, \mathcal{F}_{i-1+s}).
\end{eqnarray*}
}

\end{itemize}

\rs{Such a numerical method  is expected to be well-balanced, as long as that the sequences $\{ U^*_i \}$ of approximations of solutions
of \eqref{Cauchy2}  computed with the chosen ODE solver using a sub-mesh of step $h$ are stationary solutions of the ODE system \eqref{wbsdmeth}. 
Observe that this method would allow one to approximate the stationary solutions of \eqref{sbl} with higher order of accuracy than any other solution: they are approximated by the stationary solutions of  \eqref{wbsdmeth} with the order of accuracy of the chosen ODE solver. Moreover, the error can be made arbitrarily small by choosing $K$ large enough, i.e. using a sufficiently fine submesh.}

\subsection{\rf{Resonant problems}}
\rf{So far we have assumed for simplicity that the eigenvalues of  $J(U)$ cannot vanish. Let us consider now the general case. The principle to design well-balanced numerical methods is the same, but now in the algorithms to compute the numerical fluxes $\mathcal{F}_{i, i\pm 1/2}$ the Cauchy problem to be solved \eqref{Cauchy} cannot be written in normal form \eqref{Cauchy2} in general. As a consequence, \eqref{Cauchy} may not have a solution, or have more than one when the solution involves sonic states: in this case, the problem is said to be resonant. }

\rf{If \eqref{Cauchy} doesn't have any solution, then the data on the stencil cannot be the point values of a stationary solution and thus the numerical method
\eqref{sdmeth} will be used to update $U_i$. If it has more than one solution, a criterion is needed to select one or the other. In general this criterion may depend on the problem and on the stationary solutions to be preserved. Let us illustrate this in the case of the shallow water model.} 

\rf{For the shallow water system, given $U_i = [h_i, q_i]^T$, the solution $U^*_i = [q^*_i, h^*_i(x)]^T$ of \eqref{Cauchy} is implicitly given by
$$
q^*_i = q_i, \quad \frac{1}{2}\frac{q_i^2}{{h^*_i}^2} + gh^*_i - gH = C_i,
$$ 
with
$$
C_i = \frac{1}{2} \frac{q_i^2}{h_i^2} + gh_i - gH(x_i).
$$
Therefore, at a point $x_j$ of the stencil, one has $q_i^*(x_j) = q_i$ and $h_i^*(x_j)$ has to be a positive root of the polynomial:
$$
P_{i,j}(h) = h^3 - \left(\frac{C_i}{g} + gH(x_j)\right) h ^2+ \frac{1}{2g}q_i^2.
$$
This polynomial can have two, one, or zero positive roots. In the first case, one of the roots corresponds to a supercritical state and the other one to a subcritical state: a criterion is necessary to select one root or the other. We follow here a similar criterion to the one chosen in \cite{lopez2013} in the context of finite volume methods. A key point in this criterion is the following observation: if a smooth stationary solution $U^*$ defined  in an interval $[a, b]$ reaches a critical point at $x^*$, then necessarily $H$ has a minimum in $x^*$: see \cite{lopez2013}. In order to take into account this fact in  the selection procedure, the mesh is supposed to be such that the minimum points of $H$ belong to the set of nodes. The criterion is then as follows:}

\begin{itemize}
    \item \rf{ If all the states in the stencil $\widehat{\mathcal{S}}_i$ are supercritical and  $P_{i,j}$ has two positive roots, then the supercritical root is chosen.}
    
    \item \rf{If all the states in the stencil $\widehat{\mathcal{S}}_i$ are subcritical and  $P_{i,j}$ has two positive roots, then the subcritical root is chosen.}
    
    \item \rf{If there are subcritical and supercritical states in the stencil, the cell values can  only be the point values of a smooth stationary solution if:}
    \begin{itemize} 
    \item \rf{ $H$ has a minimum in one of the points of the stencil $x_k$.}
    \item \rf{$P_{i,k}$ has only one (critical) root.}
    \item \rf{The states at the right (resp. at the left) of $x_k$ have the same regime (sub or supercritical).}
    \end{itemize}
    \rf{If these assumptions are satisfied, then the critical root is selected  in $x_K$ (in fact, it is the only one) and if 
    $P_{i,j}$ with $j\not=k$ has two positive roots,  the supercritical one is chosen if $U_j$ is supercritical and the subcritical one is chosen if $U_j$ is subcritical.  If at least one of the above assumptions is not satisfied, then \eqref{sdmeth} is used to update $U_i$.}

\end{itemize}

\subsection{Discontinuous $H$}\label{s:Hdisc}

Let us suppose now that $H$ is a.e. differentiable with finitely many isolated jump discontinuities. In this case, the definition of weak solutions (and, in particular, of stationary solutions)  of \eqref{sbl}  becomes more difficult:  a solution $U$ is expected to be discontinuous at the  discontinuities of $H$ and, in this case,  the source term
$ S(U)H_x$ cannot be defined within the distributional framework. The source term becomes then a nonconservative product that can be defined in infinitely many different forms: see \cite{DalMaso95}.
\rs{We follow here the definition discussed in \cite{CastroPares2020}, based on the ODE system
\begin{equation}\label{jumpcondition1}
 \frac{d \ }{d\sigma}F(V) = S( V)
\end{equation}
where $\sigma$ represents the independent variable. The solutions of this ODE system may be seen as a generator of the stationary solutions of \eqref{sbl}: in effect, let us assume that $V(\sigma)$ solves \eqref{jumpcondition1}; then, given any differentiable function $H(x)$ such that
$V$ and $H$ can be composed, it can be trivially checked that
\begin{equation}\label{weaksol}
U(x) = V(H(x))
\end{equation}
is a stationary solution of \eqref{sbl}. Taking this into account, we assume here that any function of the form
\[
U^*(x) = V(H(x)),
\]
where $V$ is a solution of \eqref{jumpcondition1}, is an admissible stationary solution of \eqref{sbl} even when  $H$ has discontinuities. The idea behind this assumption is that the admissible weak solutions lie on only one integral curve of \eqref{jumpcondition1}, that is, jumps from an integral curve to another  one are forbidden at the discontinuities of $H$.}

\rs{
Let us  illustrate this in the case of the shallow water model. Although the application of the shallow water model for the simulation of a flow over a discontinuous bottom can be debatable, many authors have used it to obtain a rough simulation of the flow behavior. In any case, we consider here this application as a challenging test from the numerical analysis point of view:
once the admissible jumps at a discontinuity of $H$ have been chosen, the challenge is to design numerical methods that preserve the admissible stationary solutions. }

\rs{For this system \eqref{jumpcondition1} writes as follows: }
$$
\rs{
\left\{
\begin{array}{l}
\displaystyle  \frac{dq}{d \sigma} = 0,\smallskip\\
\displaystyle \frac{d \ }{d\sigma} \left( \frac{q^2}{h} + \frac{g}{2} h^2 \right) = g h,
\end{array}
\right.}
$$
\rs{whose solutions are implicitly given by:
$$
q = C_1, \quad \frac{1}{2}\frac{q^2}{h^2} + gh - g\sigma = C_2.
$$ 
Therefore, a function $U^* = [h^*, q^*]^T$ is considered to be an admissible weak solution if there exist two constants $C_1$ and $C_2$
such that
$$
q^*(x) = C_1, \quad \frac{1}{2}\frac{q^*(x)^2}{(h^*(x))^2} + gh^*(x) - gH(x) = C_2, \quad \forall x,
$$ 
which in particular implies that, at a discontinuity point $\bar x$ of $H$, one has:}

\rs{
\begin{equation}
\left\{
\begin{array}{l}
\displaystyle q^*(\bar x^-) = q^*(\bar x^+), \\
\displaystyle  \frac{1}{2}\frac{q^*(\bar x^-)^2}{(h^*(\bar x^-))^2} + gh^*(\bar x^-) - gH(\bar x^-) 
= \frac{1}{2}\frac{q^*(\bar x^+)^2}{(h^*(\bar x^+))^2} + gh^*(\bar x^+) - gH(\bar x^+).
\end{array}
\right.
\end{equation}
This is thus the jump condition satisfied by the admissible stationary solutions at the discontinuity points of $H$ that, in this case, may be interpreted in terms of the continuity of the mass-flow and the total energy.}

\rs{Coming back to the general case, the definition of the nonconservative product at a discontinuity point $\bar x$ of $H$ issued from our assumption is as follows:
\begin{equation}\label{dirac1}
S(U^*)H_x (\bar x) = (F(V(H(\bar x^+)) -  F(V(H(\bar x^-))))\left.\delta\right|_{x = \bar x},
\end{equation}
where $\delta$ represents Dirac's delta. 
This definition may be interpreted in terms of the choice of a particular  family of paths within the theory developed by DalMasso, LeFloch, and Murat in \cite{DalMaso95}. Moreover, it can be also interpreted in terms of the preservation of the Riemann invariants at the contact discontinuities of an extended system: see \cite{CastroPares2020} for details. }

\rs{The well-balanced numerical method \eqref{wbsdmeth}  can be easily adapted to this case just by looking for stationary solutions of the form
\eqref{weaksol} at the first stage of the algorithm that computes the numerical fluxes. More precisely, let us assume that the mesh has been designed so that all the discontinuity points of $H$ are located in an  intercell. Then, the  numerical fluxes are computed as follows:}
\begin{enumerate}

\item \rs{Look for the solution  $V^*_i (\sigma)$ of \eqref{jumpcondition1} such that:
\begin{equation}\label{firststageHd}
 V^*_i (H_i)  = U_i,
\end{equation}
where $H_i = H(x_i)$.}

\item \rs{Define
    $$ \mathcal{F}_j = F(U_j) - F(V^*_i(H(x_j))), \quad j = i-1-r,\dots, i+s$$    }

\item  \rs{Compute
\begin{eqnarray*}
\widehat{\mathcal{F}}_{i,i+1/2} & =&  \rs{\mathcal{R}}(\mathcal{F}_{i-r}, \dots, \mathcal{F}_{i+s}), \\
\widehat{\mathcal{F}}_{i,i-1/2} & = & \rs{\mathcal{R}}(\mathcal{F}_{i-1-r}, \dots, \mathcal{F}_{i-1+s}).
\end{eqnarray*}
}

\end{enumerate}

\rs{ The numerical fluxes issued from the algorithm are formally consistent with the definition \eqref{dirac1} for the nonconservative product.} Although consistency is not enough to guarantee the convergence to the right weak solution of nonconservative systems (see \cite{ParesMunoz}, \cite{CLFMP2008}) the numerical tests in Section \ref{ss:tests} show that the numerical methods capture the correct weak solutions.

In order to compare the behaviour of the different methods in the presence of a discontinuity of $H$, for WENO methods \eqref{sbl} the Dirac delta issued from a discontinuity of $H$ will be approached by
\begin{equation}\label{dirac2}
S_{I-1/2}\frac{H^+(x_{I-1/2}) - H^-(x_{I-1/2})}{\Delta x}\left.\delta \right|_{x = x^*},
\end{equation}
where $S_{I-1/2}$ is some intermediate value of $S(U)$ at the discontinuity. An upwind treatment of the singular source term is then used, so that the numerical method for the neighbor nodes writes as follows:
\begin{eqnarray}\label{sdmethsstsw}
& & \frac{d U_{I-1}}{dt} + \frac{1}{\Delta x} \left( \widehat{{F}}_{I - 1/2} - \widehat{{F}}_{I-3/2} \right)  = 
S(U_{I-1})H_x(x_{I-1}) + S^-_{I-1/2}, \\
& & \frac{d U_{I}}{dt} + \frac{1}{\Delta x}   \left( \widehat{{F}}_{I + 1/2} - \widehat{{F}}_{I-1/2} \right)  = 
S(U_{I})H_x(x_{I}) + S^+_{I-1/2}, \nonumber
\end{eqnarray}
where
$$
S^\pm_{I-1/2} = P^\pm_{i+1/2} S_{I-1/2}\frac{H(x_I) - H(x_{I-1})}{\Delta x}.
$$
Here, $P^+_{i+1/2} $ are the projection matrices given by \eqref{projectors}. 
It will be seen in Section \ref{ss:tests} that these methods do not converge to the assumed weak solutions. Moreover, the behaviour of the numerical methods at the discontinuity of $H$ depends both on $\Delta x$ and the chosen intermediate state. 

In the case of the methods that only preserve one stationary solution $U^*$, the source term in the neighbor nodes of the discontinuity will be computed as follows:
\begin{eqnarray}\label{sdmethwb1sw}
& & \frac{d U_{I-1}}{dt} + \frac{1}{\Delta x}   \left( \widehat{\mathcal{F}}_{I - 1/2} - \widehat{\mathcal{F}}_{I-3/2} \right)  = 
(S(U_{I-1})- S(U^*(x_{I-1})))H_x(x_{I-1}) + S^-_{I-1/2}- S^{*-}_{I-1/2}, \\
& & \frac{d U_{I}}{dt} + \frac{1}{\Delta x}   \left( \widehat{\mathcal{F}}_{I + 1/2} - \widehat{\mathcal{F}}_{I-1/2} \right)  = 
(S(U_{I})- S(U^*(x_I)))H_x(x_{I}) + S^+_{I-1/2}- S^{*+}_{I-1/2}, \nonumber
\end{eqnarray}
with
$$
S^{*\pm}_{I-1/2} = P^\pm_{i+1/2} S^*_{I-1/2}\frac{H(x_I) - H(x_{I-1})}{\Delta x}.
$$
Clearly, if $U_i = U^*(x_i)$ for $i = I-1, I$ the right-hand sides vanish and the approximation of the Dirac mass is given again 
by \eqref{dirac1} but, if it is not the case, a combination of \eqref{dirac1} and \eqref{dirac2} is used. Therefore, this method is only consistent with the definition of weak solution when the stationary solution $U^*$ is not perturbed at the neighbour nodes of the discontinuity, as it will be seen in Section \ref{ss:tests}.

\rf{
\subsection{Multidimensional problems}\label{ss:2d}
Although the extension to multidimensional problems of the methods introduced here is out of the scope of this paper, let us briefly discuss about it. Let us focus on two-dimensional  problems:
\begin{equation}\label{sbl2d}
U_t + F_1(U)_x + F_2(U)_y = S_1(U) H_x + S_2(U)H_y.
\end{equation}
In principle, the idea developed here can be applied: if $U^*$ is a stationary solution satisfying
\begin{equation}\label{initialcond}
U^*(x_i, y_j) = U_{i,j},
\end{equation}
then the source term can be discretized as follows:
\[
S_1(U_{i,j}) H_x(x_{i,j}) + S_2(U_{i,j})H_y(x_{i,j}) = \partial_x F_1(U^*(x_{i,j})) + \partial_y F_2(U^*(x_{i,j})),
\]
where $x_{i,j}$ are the nodes of a Cartesian mesh of step sizes $\Delta x$, $\Delta y$. Then, the reconstruction operator would be applied to
the differences
\begin{equation}\label{sbldss}
F_1(U_{i,j}) - F_1(U^*(x_{i,j})), \quad F_2(U_{i,j}) - F_2(U^*(x_{i,j}))
\end{equation}
to compute the numerical fluxes in both directions.}

\rf{Of course, the main difficulty comes from the fact that now the problem to be solved for finding $U^*$ is a PDE system:
\[
F_1(U^*)_x + F_2(U^*)_y = S_1(U^*) H_x + S_2(U^*)H_y,
\]
that is much more difficult to solve either exactly or numerically  than an ODE system. Moreover in this case \eqref{initialcond} does not determine a stationary solution:  there may exist infinitely many stationary solutions satisfying this equality. The only way to extend the numerical method introduced  here in order to preserve any stationary solution would be to replace  \eqref{initialcond} by some adequate boundary conditions in the stencil that guarantee the uniqueness of solution and then solve numerically the corresponding boundary value problem. Nevertheless this program is far from being easy: the selected  boundary conditions must take into account the character of the PDE satisfied by the stationary solutions (hyperbolic, elliptic, mixed, etc.) that can change from  one stencil to another.}

\rf{On  the other hand, the extension to 2d problems of the numerical methods that preserve a given family of known stationary solution
introduced in Section \ref{ss:wb1} is straightforward: if the family depends on $k$ parameters
\[
U^*(x,y; C_1, \dots, C_k)
\]
with $k \leq N$, the first step of the algorithm to compute the numerical fluxes would be:}

\rf{Find $C^{i,j}_1, \dots, C^{i,j}_k $ such that
\[
u^*_{j_l}(x_i,y_j;C^{i,j}_1, \dots,  C^{i,j}_k ) = u_{i,j,l}, \quad l=1, \dots, k,
\]
where $j_1, \dots, j_k$ is a predetermined set of $k$ indices.}

\rf{In particular, in the case of the shallow water model, the extension of the numerical methods that preserve water-at-rest solutions to 2d is straightforward: it is enough to apply the computation of the numerical fluxes shown in Section \ref{ss:wb1} in  both directions.}

\section{Numerical tests}\label{ss:tests}
In this section we apply the numerical methods introduced in Sections 2 and 3 to a number of test cases with $H$ continuous or discontinuous.
In the first two subsections we consider two scalar problems: the linear transport equation and Burgers equations with source terms. Three families
of methods based on WENO reconstructions of order $p$ will be compared:
\begin{itemize}
    \item WENO$p$: methods of the form \eqref{sdmeth}.
\item WBWENO$p$: methods of the form \eqref{wbsdmeth} that preserve any stationary state.
\item WB1WENO$p$: methods of the form \eqref{wb1s} that preserve only one given stationary state.
\end{itemize}
Nevertheless, in many test cases the results obtained with WBWENO$p$ and WB1WENO$p$ are indistinguishable: in those cases, only the results corresponding to 
WBWENO$p$ will be shown. In  all cases, the global Lax-Friedrichs flux-splitting approach is used for WENO implementation and the third order TVD-RK3 method is applied for the time discretization: see \cite{Gottlieb98}.
The CFL parameter is set to 0.5.

\subsection{A linear problem}
We consider the linear scalar  problem
\begin{equation}\label{testlinear}
u_t + u_x = u H_x.
\end{equation}
 In this case, \eqref{jumpcondition1} reduces to
\begin{equation}\label{odetest1}
\frac{dv}{d\sigma} = v,
\end{equation}
whose solutions are
$$
V(\sigma) = C e^\sigma, \quad C \in \mathbb{R}.
$$
The stationary solutions of \eqref{testlinear} for any given $H$ are thus given by:
\begin{equation}\label{sstest1}
u^*(x) = C e^{H(x)}, \quad C \in \mathbb{R}.
\end{equation}
The solution of \eqref{firststageHd} is thus
$$
v^*_i (\sigma) = u_i e^{\sigma- H_i}.
$$
Therefore, well-balanced methods are based on the reconstructions of 
$$
f_j =  u_j - \rf{ u_i e^{(H(x_j)-H_i)}}, \quad j = i- r, \dots, i + s.
$$

\subsubsection{Order test}\label{ss_orderlHc}
Let us consider \eqref{testlinear} with
$$
H(x) = x.
$$
It can be easily checked that the solution of \eqref{testlinear} with initial condition:
$$
u(x,0) =  u_0(x),\quad x \in \mathbb{R}
$$
is given by
$$
u(x,t) = e^tu_0(x-t), \quad x \in \mathbb{R}.
$$
Let us consider the initial condition:
\begin{equation}\label{test1ic}
u_0(x) = \begin{cases} 
0 & \text{if $x < 0$,}\\
p(x) & \text{if $0 \leq x \leq 1$,} \\
1 & \text{otherwise,}
\end{cases}
\end{equation}
where $p$ is the 11th degree polynomial
$$
p(x) = x^6\left(\sum_{k = 0}^5 (-1)^k \left( \begin{array}{c} 5 + k \\ k \end{array} \right) (x -1)^k \right)
$$
such that
$$
p(0) = 0, \quad p(1) = 1, \quad p^{k}(0) = p^{k}(1) = 0, \quad k =1, \dots, 5
$$
see Figure \ref{fig-orderlHc-1}.

\begin{figure}
\centering
\includegraphics[width=0.4\textwidth]{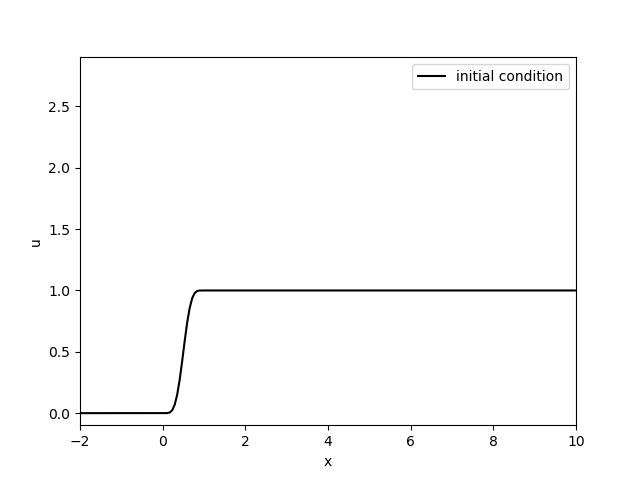}
\includegraphics[width=0.4\textwidth]{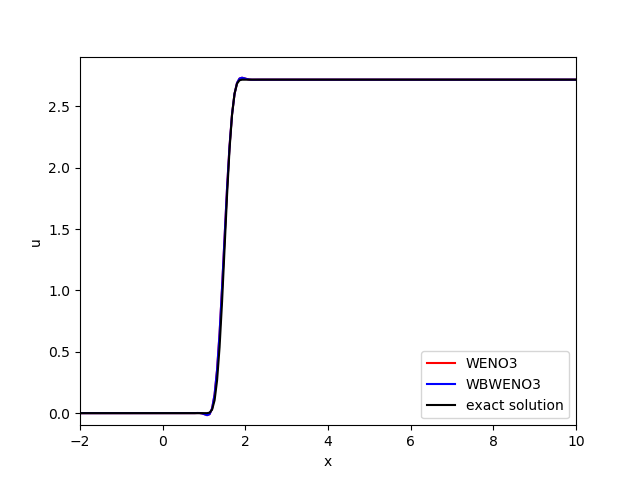}
\caption{\label{fig-orderlHc-1} Test 5.1.1: initial condition (left). Exact solution and numerical solution obtained with WBWENO3 and WBWENO5 at time $t = 1$ using a mesh of 200 cells}
\end{figure}

We solve \eqref{testlinear} with initial condition \eqref{test1ic} with well-balanced and non well-balanced third order methods in the interval $[-2, 10]$. Free boundary conditions based on the use of ghost cells are used at both extremes. Table \ref{table1} shows the $L^1$-errors and  the empirical order of convergence corresponding to WENO$p$ and WBWENO$p$, $p =3, 5$. As can be seen, both methods are of the expected order and the errors corresponding to methods of the same order are almost identical. In order to capture the expected order, the smooth indicators of the WENO reconstruction have been set  to 0 and $\Delta t = \Delta x^{5/3}$ has been chosen for the fifth order methods.

\begin{table}[H]
\centering
\begin{tabular}{|c|c|c|c|c|c|c|c|c|} \hline
 & \multicolumn{2}{|c|}{WENO3}  & \multicolumn{2}{|c|}{WBWENO3} & \multicolumn{2}{|c|}{WENO5}  & \multicolumn{2}{|c|}{WBWENO5} \\
Cells & Error  &Order&Error   & Order & Error  &Order & Error  &Order \\
\hline
100 & 1.000E-1 & -& 1.023E-1 & - & 4.0902E-2 & -& 4.0910E-2 & - \\
200  & 2.053E-2 &  2.28& 2.084E-2 & 2.29 & 2.4404E-3 &  4.06  & 2.4407E-3 & 4.06\\
400 & 2.978E-3 & 2.78 &  3.019E-3& 2.78 & 9.1307E-5 & 4.74  & 9.1315E-5 & 4.74\\
800 & 3.815E-4 & 2.96 &  3.867E-4 & 2.96 & 3.0118E-6  & 4.92 & 3.0121E-6 & 4.92\\
1600 &  4.788E-5 & 2.99 & 4.855E-5 &2.99 & 9.4849E-8 & 4.98  & 9.4857E-8 & 4.98
 \\
  \hline
\end{tabular}
\caption{Test \ref{ss_orderlHc}. Errors in $L^1$ norm and convergence rates for WB$p$ and WBWENO$p$, $p = 3,5$ at time $t = 1$.} \label{table1}
\end{table}

\subsubsection{A moving discontinuity linking two stationary solutions} \label{ss_mslHc}
Next, we consider \eqref{testlinear} with again $H(x) = x$, and initial condition
$$
u_0(x) = \begin{cases}
4e^x & \text{if $x < 0$,}\\
e^x & \text{otherwise.}
\end{cases}
$$
The solution consists of a discontinuity linking two steady states that travels at speed 1:
$$
u(x, t) = \begin{cases}
4e^x & \text{if $x < t$,}\\
e^x & \text{otherwise.}
\end{cases}
$$
Figure \ref{fig-mslHc-1} shows the exact and the numerical solutions at time $t = 1$ obtained with WBWENO3 and WENO3 (left) and a zoom of the differences of the numerical and the exact solution at the same time (right). It can be observed that the stationary states at both sides of the discontinuity are better captured with the well-balanced method. For fifth order methods the differences are lower, but still noticeable: see Figure \ref{fig-mslHc-2}. The results obtained with WBWENO$p$, and the WB1WENO$p$, $p =3,5$ that only preserve the stationary solution $u^*(x) = 4 e^x$, are indistinguishable.
\begin{figure}
\centering
\includegraphics[width=0.4\textwidth]{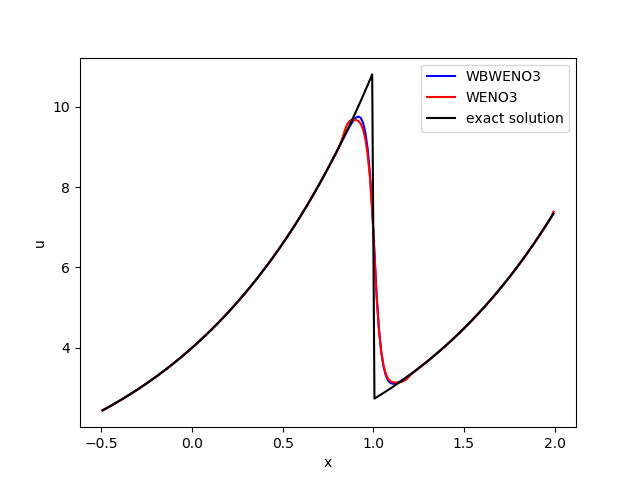}
\includegraphics[width=0.4\textwidth]{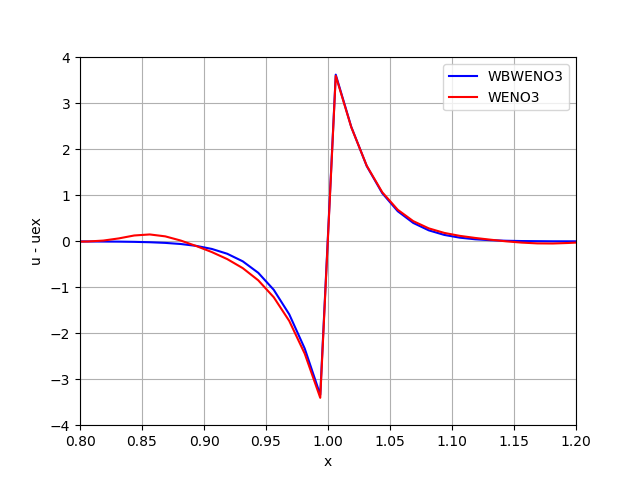}
\caption{Test \ref{ss_mslHc}: exact solution and numerical solutions obtained with WBWENO3 and WENO3 using a mesh of 200 cells, $t = 1$ (left); zoom of the differences between the numerical and the exact solutions (right)} \label{fig-mslHc-1}
\end{figure}

\begin{figure}
\centering
\includegraphics[width=0.4\textwidth]{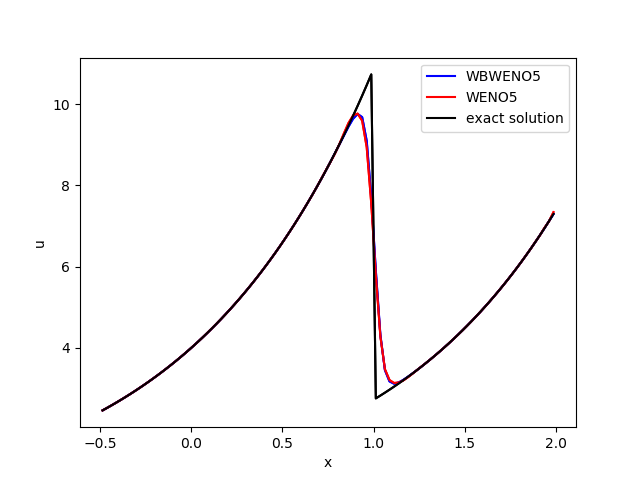}
\includegraphics[width=0.4\textwidth]{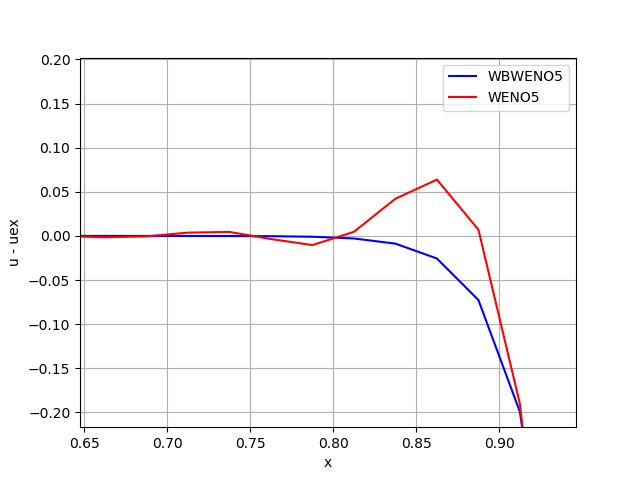}
\caption{Test \ref{ss_mslHc}: exact solution and numerical solutions obtained with WBWENO5 and WENO5 using a mesh of 100 cells, $t = 1$ (left); zoom of the differences between the numerical and the exact solutions (right)} \label{fig-mslHc-2}
\end{figure}

\subsection{Burgers' equation with source term}
We consider next the scalar equation
\begin{equation}\label{testburgers}
u_t + f(u)_x = s(u)H_x,
\end{equation}
with
$$
f(u) = \frac{1}{2}u^2, \quad s(u) = u^2.
$$

The stationary solutions are also given by \eqref{sstest1}. 
The reconstruction operator has to be applied in this case to
$$
f_j = \frac{ u_j^2}{2} - \frac{u_i^2 \rf{e^{2(H(x_j)-H_i)}}}{2}, \quad j = i- r, \dots, i + s.
$$

\subsubsection{Preservation of a stationary solution with smooth $H$}\label{ss_stBHc}
In this test case we consider $H(x) = x$ and the stationary solution
\begin{equation}
u(x) = e^x.
\end{equation}
Let us solve  \eqref{testburgers}  taking this stationary solution as initial condition in the interval 
$[-1, 1]$. As boundary conditions, the value of the stationary solution is imposed at the ghost cells.
Figures  \ref{fig-stBHc-1} and \ref{fig-stBHc-2} show the differences between the stationary solution and the numerical solutions obtained at time $t = 8$ with WENO$p$ and WBWENO$p$, $p = 3, 5$ using a 200-cell mesh: the well-balanced methods capture the stationary solution with machine accuracy. This is confirmed by Tables 
\ref{table_stBHc-1} and \ref{table_stBHc-2} that show the $L^1$-errors and  the empirical order of convergence corresponding to WENO$p$ and WBWENO$p$, $p =3, 5$.
\begin{figure}
\centering
\includegraphics[width=0.4\textwidth]{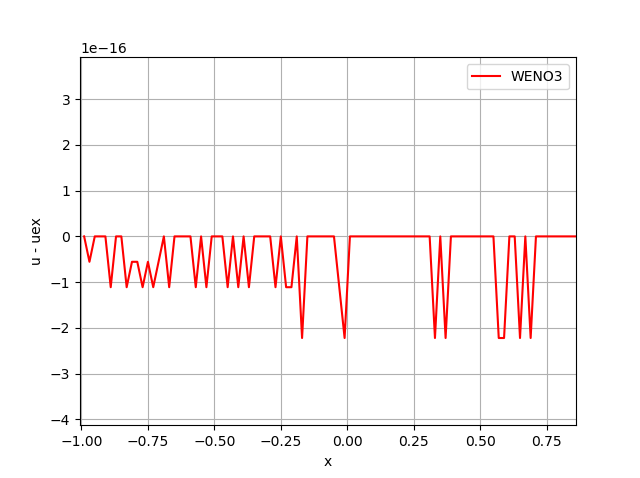}
\includegraphics[width=0.4\textwidth]{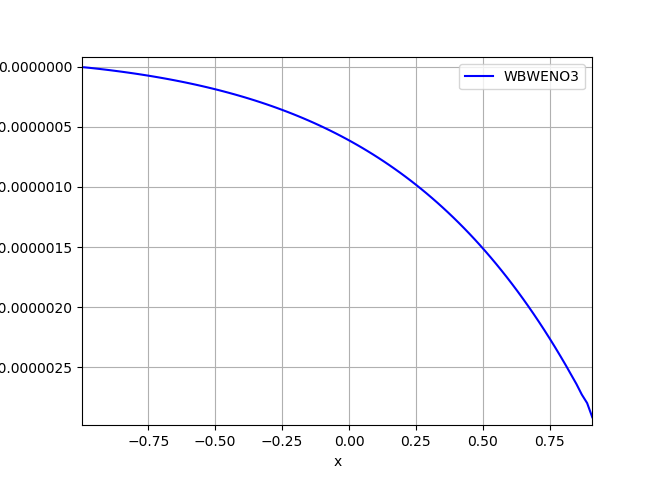}
\caption{Test \ref{ss_stBHc}:  zoom of the differences between the numerical solutions at time $t = 8$ and the stationary solution using a 200-cell mesh. Left: WBWENO3. Right: WENO3} \label{fig-stBHc-1}
\end{figure}
\begin{figure}
\centering
\includegraphics[width=0.4\textwidth]{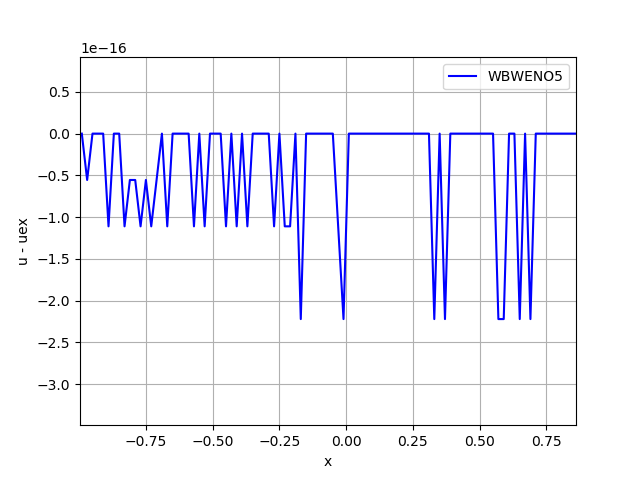}
\includegraphics[width=0.4\textwidth]{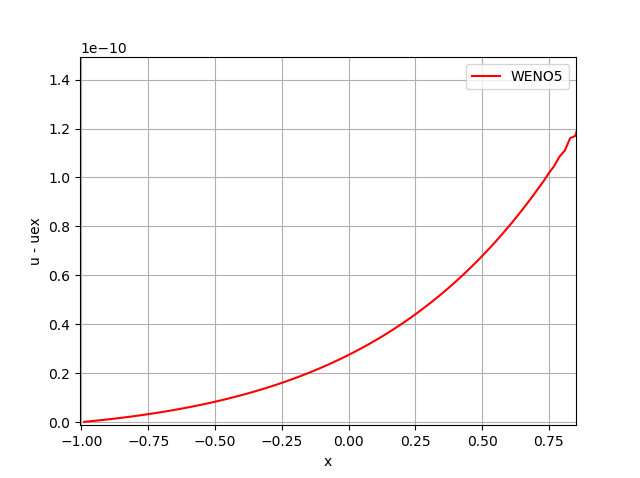}
\caption{Test \ref{ss_stBHc}:  zoom of the differences between the numerical solutions at time $t = 8$ and the stationary solution using a 200-cell mesh. Left: WBWENO5. Right: WENO5} \label{fig-stBHc-2}
\end{figure}

\begin{table}[H]
\centering
\begin{tabular}{|c|c|c|c|} \hline
 & \multicolumn{2}{|c|}{WENO3}  & {WBWENO3} \\
Cells & Error  &Order&Error  \\
\hline
100 & 1.9044E-06 & - & 8.9928E-17\\
200 & 2.4762E-07 & 2.94 & 1.4543E-16\\
400 & 3.1550E-08 & 2.97 & 1.5304E-14\\
800 & 3.9817E-09 & 2.98 & 1.6560E-14  \\
  \hline
\end{tabular}
\caption{Test \ref{ss_stBHc}. Errors in $L^1$ norm and convergence rates for WB3 and WBWENO3 at time $t = 8$.} \label{table_stBHc-1}
\end{table}

\begin{table}[H]
\centering
\begin{tabular}{|c|c|c|c|} \hline
 & \multicolumn{2}{|c|}{WENO5}  & {WBWENO5} \\
Cells & Error  &Order&Error  \\
\hline
20 & 7.7695E-07 & - &   2.2759e-16\\
40 & 3.5170E-09 & 7.78 &   1.5543e-16\\
80 & 2.0005E-10 & 4.13 &   1.1657e-16\\
160 & 1.0352E-11& 4.27&  2.9559e-16\\
  \hline
\end{tabular}
\caption{Test \ref{ss_stBHc}. Errors in $L^1$ norm and convergence rates for WENO5 and WBWENO5 at time $t = 8$.} \label{table_stBHc-2}
\end{table}


\subsubsection{Preservation of a stationary solution with oscillatory smooth $H$}\label{ss_stBHo}
Let us consider now \eqref{testburgers} with a function $H$ that has an oscillatory behavior:
\begin{equation} \label{H_osc}
H(x) = x + 0.1 \sin(100x),
\end{equation}
(see Figure \ref{fig-stBHo-1}). We consider again the interval $[-1,1]$ and we take as initial condition the stationary solution
$$
u(x) = e^{x + 0.1 \sin(100x)},
$$
(see Figure \ref{fig-stBHo-1}) and, as boundary conditions,  the value of this stationary solution is also imposed at the ghost cells.

\begin{figure}
\centering
\includegraphics[width=0.4\textwidth]{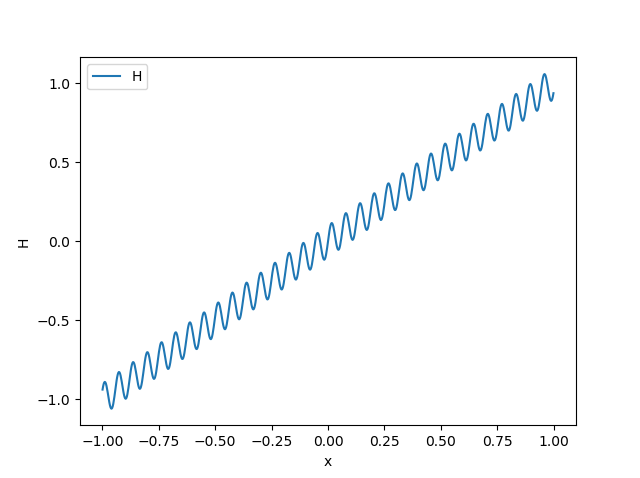}
\includegraphics[width=0.4\textwidth]{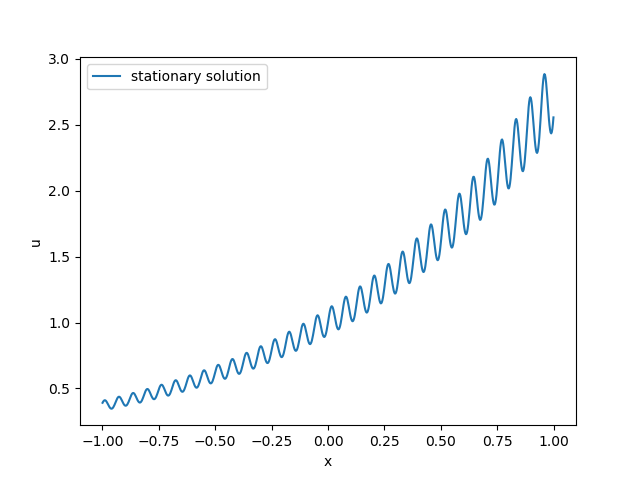}
\caption{Test \ref{ss_stBHo}:  graph of the function $H$ (left) and stationary solution (right)} \label{fig-stBHo-1}
\end{figure}

\begin{figure}
\centering
\includegraphics[width=0.4\textwidth]{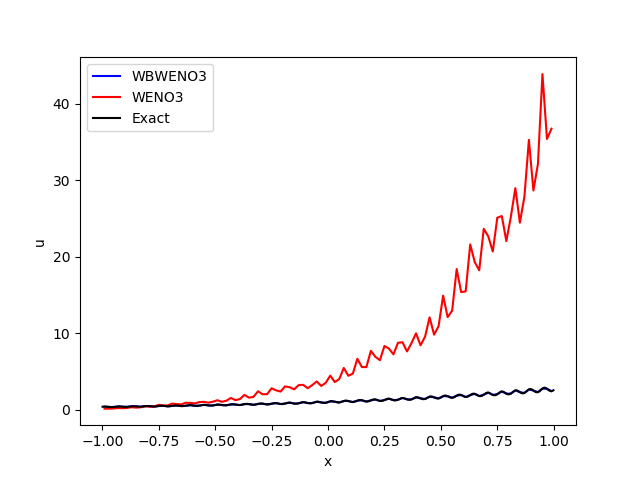}
\includegraphics[width=0.4\textwidth]{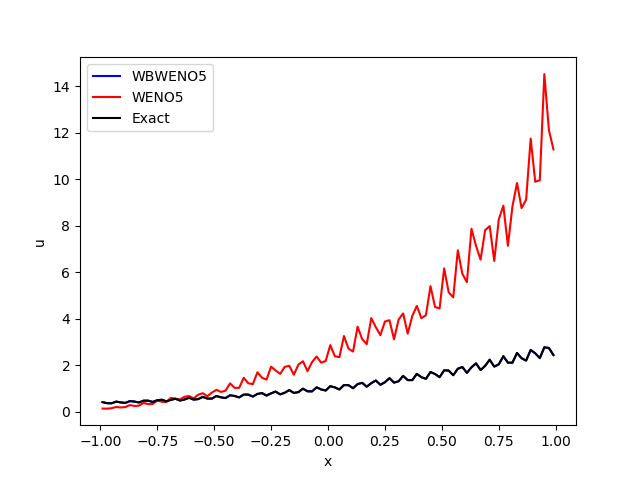}
\caption{Test \ref{ss_stBHo}:  exact solution and numerical solutions obtained at time $t = 1$ and a mesh of 100 cells. Left: WBWENO3, WENO3. Right: WBWENO5,  WENO5 \rf{(the graphs corresponding to the stationary solution and the numerical solutions obtained with WBWENO$p$ are  indistinguishable)}.} \label{fig-stBHo-2}
\end{figure}

We consider a 100-cell mesh, so that the period of the oscillations is close to $\Delta x$.
Figure \ref{fig-stBHo-2} shows the numerical solutions at time $t = 1$ corresponding to WBWENO$p$, WENO$p$, $p = 3,5$:
while the well-balanced methods preserve the stationary solution with machine precision \rf{(the graphs corresponding to the stationary solution and the numerical solutions obtained with WBWENO$p$ are  indistinguishable)}, the non well-balanced methods give a wrong numerical solution. Of course, they give more accurate solutions if the mesh is refined: see next paragraph, where the reference solution is computed using WENO3. 

\subsubsection{Perturbation of a stationary solution with oscillatory smooth $H$}\label{ss_pBHo}
We consider again Burgers' equation \eqref{testburgers} with $H$ as in \eqref{H_osc}, and an initial condition that is the stationary solution approximated in the previous test with a small perturbation
$$
u_0(x) = e^{x + 0.1 \sin(100x)} + 0.1 e^{-200(x+5)^2},
$$
(see Figure \ref{fig-pBHo-1}). If a mesh with 100 cells is used, the non well-balanced methods are unable to follow the evolution of the perturbation, since the numerical errors observed in the previous test are much larger than the perturbation. Let us see what happens when well-balanced methods are used: Figure \ref{fig-pBHo-2} shows the numerical solutions obtained with WBWENO3 and WBWENO5 and a reference solution computed with WENO3 using a mesh of 5000 cells at time $t = 1$. As it can be seen, both methods are able to follow the evolution of the perturbation. 
\begin{figure}
\centering
\includegraphics[width=0.4\textwidth]{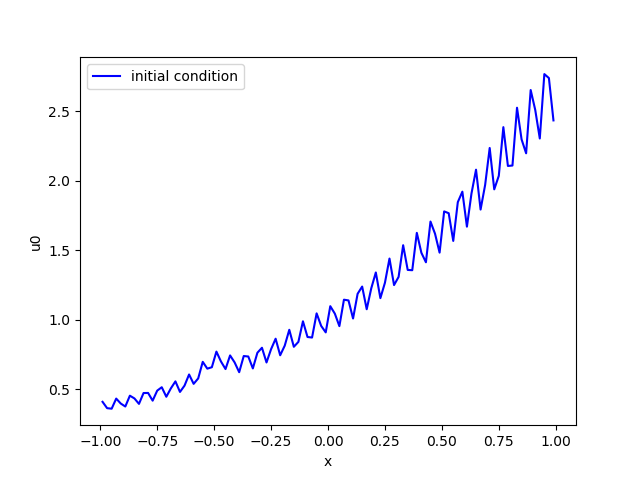}
\includegraphics[width=0.4\textwidth]{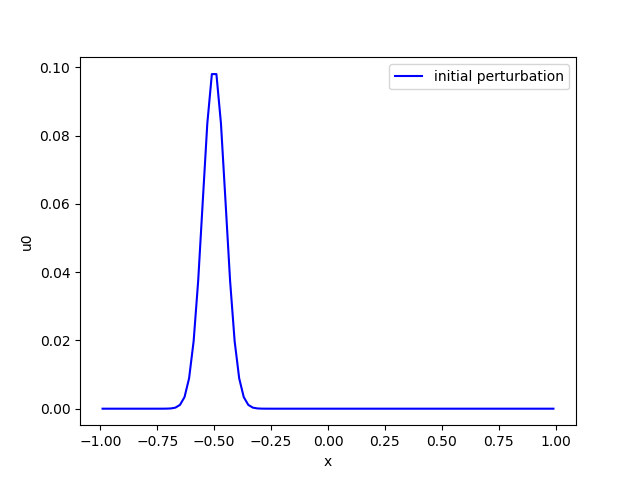}
\caption{Test \ref{ss_pBHo}:  initial condition. Left: graph. Right: difference with the stationary solution} \label{fig-pBHo-1}
\end{figure}
\begin{figure}
\centering
\includegraphics[width=0.4\textwidth]{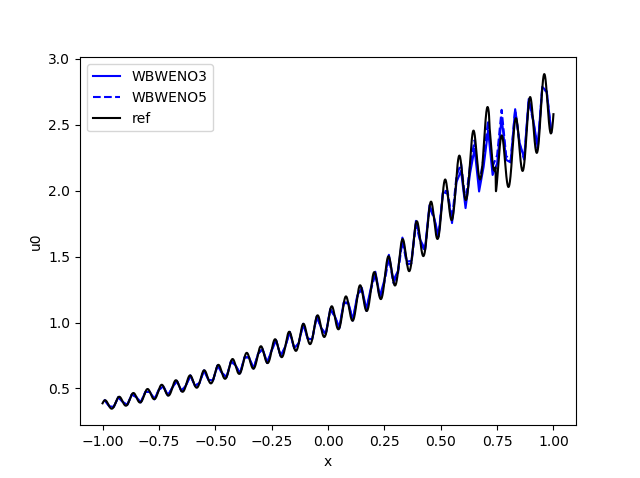}
\includegraphics[width=0.4\textwidth]{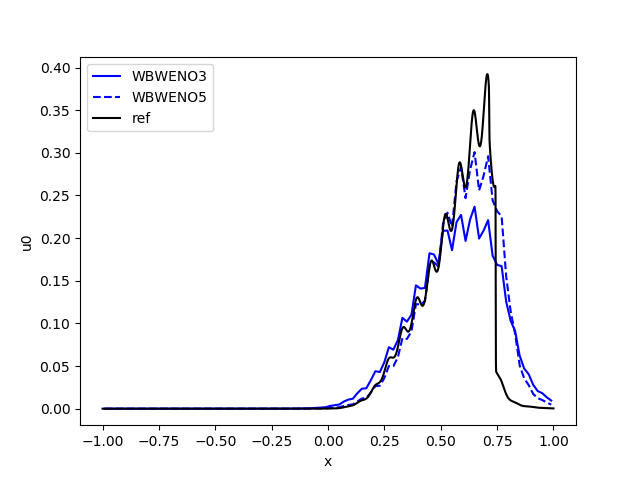}
\caption{Test \ref{ss_pBHo}:  reference and numerical solutions obtained with WBWENO3 and WBWENO5 at time $t = 1$ and a mesh of 100 cells. Left: graphs. Right: difference with the stationary solutions} \label{fig-pBHo-2}
\end{figure}

\subsubsection{Preservation of a stationary solution with piecewise continuous $H$}\label{ss_stBHd}
Let us consider now \eqref{testburgers} with a piecewise continuous function $H$:
\begin{equation} \label{H_disc}
H(x) = \begin{cases} 0.1 x & \text{ if $x \leq 0$;} \\
0.9 +x & \text{otherwise;}
\end{cases}
\end{equation}
(see Figure \ref{fig-stBHd-1}). We consider again the interval $[-1,1]$ and we take as initial condition the stationary solution
\begin{equation}\label{stsolBHd}
u(x) = \begin{cases} 
e^{0.1 x} & \text{ if $x \leq 0$;} \\
e^{0.9 +x} & \text{otherwise;}
\end{cases}
\end{equation}
(see Figure \ref{fig-stBHd-1}) and, as boundary conditions,  the value of this stationary solution is also imposed at the ghost cells. WBWENO$p$, $p=3,5$ preserve the stationary solution again with machine accuracy: see Table \ref{table_stBHD-1}.

\begin{table}[ht!]
\centering
\begin{tabular}{|c|c|c|} \hline
Cells & WBWENO3  &WBWENO5  \\
\hline
100 & 7.4984E-15 &  5.3790E-14\\
200 & 1.5432E-16 & 1.7763E-16\\
300 & 2.1464E-16 &  4.1611E-15\\
  \hline
\end{tabular}
\caption{Test \ref{ss_stBHd}. Errors in $L^1$ norm for WBWENO3 and WBWENO5 at time $t = 1$.} \label{table_stBHD-1}
\end{table}



\begin{figure}
\centering
\includegraphics[width=0.4\textwidth]{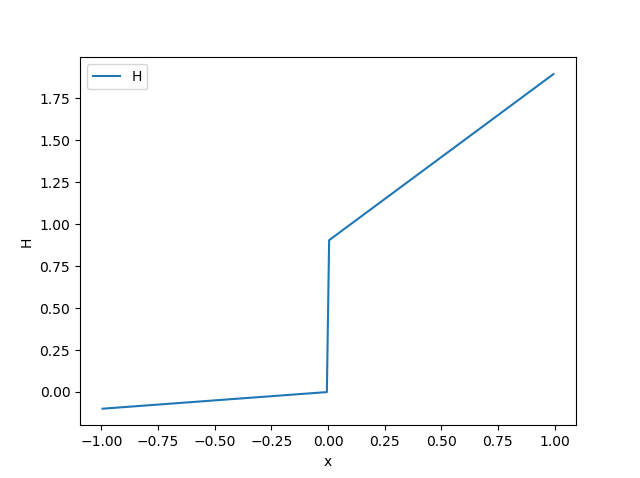}
\includegraphics[width=0.4\textwidth]{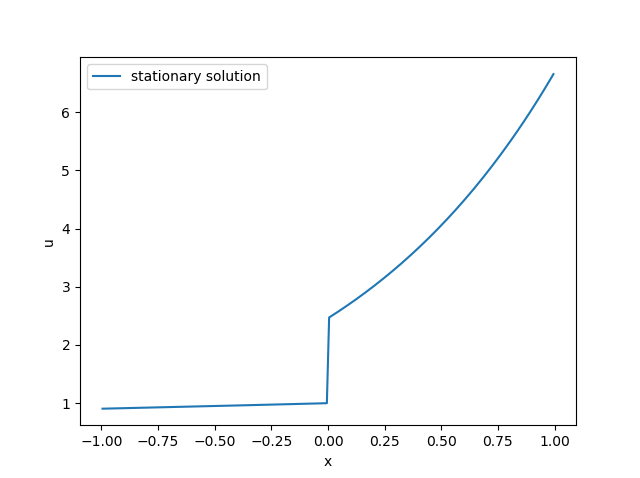}
\caption{Test \ref{ss_stBHd}:  graph of the function $H$ (left) and stationary solution (right)} \label{fig-stBHd-1}
\end{figure}

For WENO methods, the non-conservative product appearing at the source term is discretized by \eqref{sdmethsstsw}
with  two different definitions of $s_{I-1/2}$: a centered one
\begin{equation}\label{Smean}
s_{I-1/2} = s(0.5(u_{I-1} + u_I))
\end{equation}
or an upwind one
\begin{equation}\label{Supw}
   s_{I-1/2} = \left( \frac{ 1 + sign(u_{I-1/2})}{2}\right) s(u_{I-1}) + 
   \left( \frac{ 1 - sign(u_{I-1/2})}{2}\right)s( u_{I}).
\end{equation}
Here, $I-1/2$ is the index of the intercell is located and $u_{I-1/2}$ is the arithmetic mean of $u_{I-1}$ and $u_I$.

In Figure \ref{fig-stBHd-2}  we compare the exact solution with the numerical solutions obtained at time $t = 1$ obtained with 
WBWENO3 using a 300-cell mesh (its graph and the one of the exact solution are identical at the scale of the figure) and with WENO3 or WENO5 using different implementations:
\begin{itemize}
    \item WENO3-UPW1: WENO3 with upwind implementation and \eqref{Supw};
    \item WENO3-UPW2: WENO3 with upwind implementation and \eqref{Smean};
    \item WENO3-LF:  WENO3 with LF implementation and \eqref{Supw};
    \item WENO5-LF: WENO5 with LF implementation and \eqref{Supw}.
\end{itemize}
As it can be observed, the results of the WENO methods depend on the chosen implementation, on the numerical definition of the source term, and on the order. Moreover, these differences remain as $\Delta x$ tends to 0. Note that in this example, unlike those discussed so far, we have a discontinuous $H$. Therefore the term $S(U)H_x$ is not uniquely defined, and the fact that different methods converge to different solutions should not be surprising: this is in good agreement with the difficulties of convergence of finite difference methods to nonconservative systems (see \cite{CLFMP2008}). 
\begin{figure}
\centering
\includegraphics[width=0.4\textwidth]{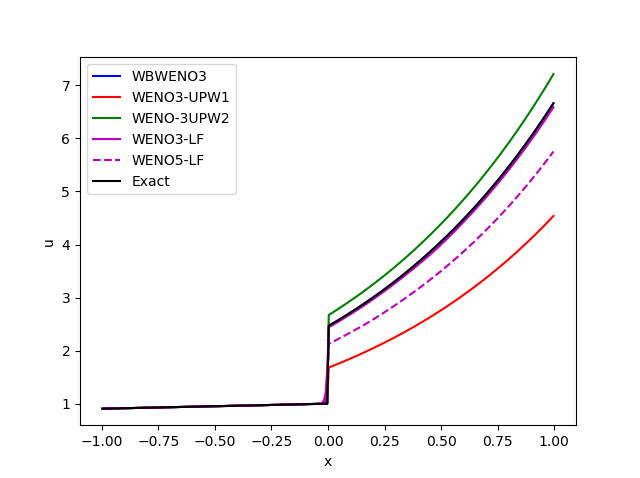}
\includegraphics[width=0.4\textwidth]{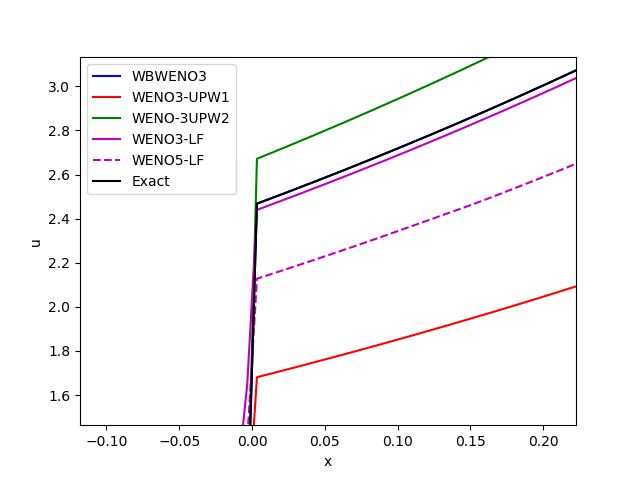}
\caption{Test \ref{ss_pBHo}: exact and numerical solutions obtained 
at time $t = 1$ with WBWENO3,WENO3-UPW1, WENO3-UPW2,WENO3-LF, WENO5-LF. Left: global view. Rigth: zoom close to the discontinuity}
 \label{fig-stBHd-2}
\end{figure}

\subsubsection{Perturbation of a stationary solution with piecewise continuous $H$}\label{ss_pBHd}
We consider again \eqref{testburgers} with \eqref{H_disc} and an initial condition that is the stationary solution approximated in the previous test with a small perturbation
$$
\tilde u_0(x) = u(x) + 0.3 e^{-200(x+5)^2},
$$
where $u$ is given by \eqref{stsolBHd}: see Figure \ref{fig-pBHd-1}. Again, WENO$p$, $p = 3,5$ are unable to follow the evolution of the perturbation, since they are not able to preserve the stationary solution.  Figure \ref{fig-pBHd-1} shows the initial condition and the numerical solutions obtained with WBWENO$p$ and WB1WENO$p$, $p =3,5$ with a mesh of 300 cells together with a reference solution computed with WBWENO3 using a mesh of 5000 cells at time $t = 0.5$. As it can be seen, while WB1WENO$p$ preserves the stationary solution (the results obtained in the previous test are indistinguishable from those obtained with WBWENO$p$), once the perturbations arrive to the discontinuity, the stationary solution is no longer preserved after its passage: see the discussion in Section \ref{s:Hdisc}.
\begin{figure}
\centering
\includegraphics[width=0.4\textwidth]{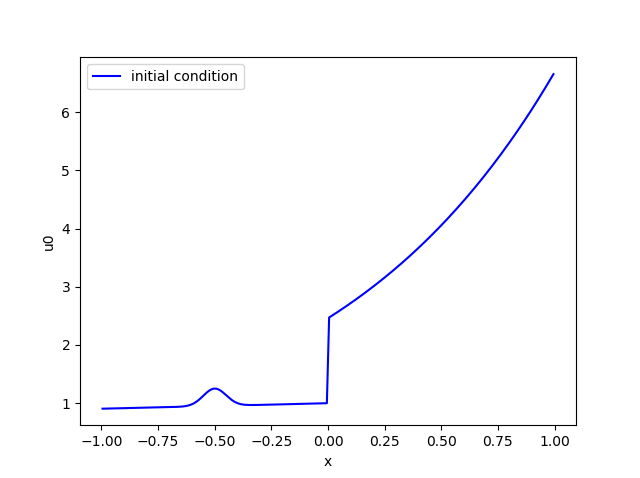}
\caption{Test \ref{ss_pBHd}:  initial condition. } \label{fig-pBHd-1}
\end{figure}

\begin{figure}
\centering
\includegraphics[width=0.4\textwidth]{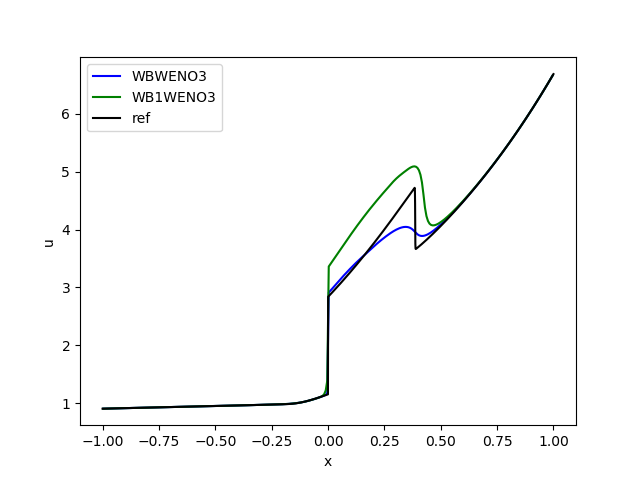}
\includegraphics[width=0.4\textwidth]{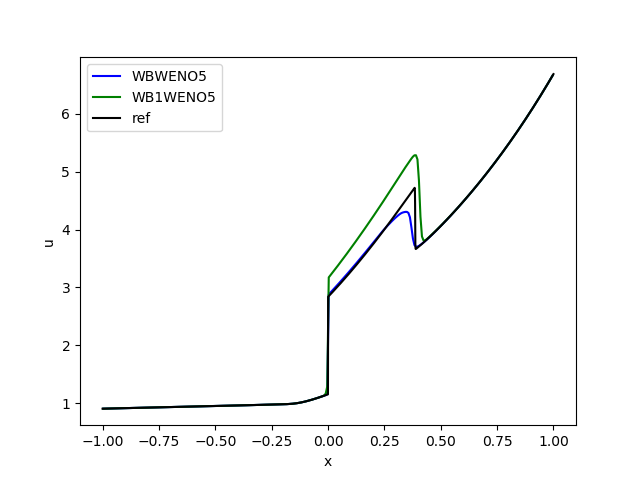}
\caption{Test \ref{ss_pBHd}:  numerical solutions obtained with WBWENO$p$ and WB1WENO$p$ at time $t = 0.5$ and a mesh of 300 cells: left $p = 3$, right $p = 5$ } \label{fig-pBHd-2}
\end{figure}

\subsection{Shallow water equations}
Five different numerical methods are considered for the shallow water system \eqref{sbl}:
\begin{itemize}
\item WENO$p$: methods of the form \eqref{sdmeth}.
\item WBWENO$p$: methods of the form \eqref{wbsdmeth} that preserve any stationary state.
\item WB1WENO$p$: methods of the form \eqref{wb1s} that preserve only one given stationary state.
\item WBWARWENO$p$: methods that preserve water at rest stationary solutions described in Section \ref{ss:wb1}.
\item WBMCWENO$p$: methods that preserve every stationary solutions and the total mass described in Section \ref{ss:mc}.
\end{itemize}

\subsubsection{Preservation of a subcritical stationary solution}\label{ss_stsubSWHc}
We consider the shallow water system with the bottom depth given by
 \begin{equation}\label{bump2}
 H(x) =  \begin{cases}
  -0.25(1 + \cos(5\pi x)) & \text{if $ -0.2 \leq x \leq 0.2$;}\\
  0 & \text{otherwise;}
 \end{cases}
 \end{equation}
and we take as initial condition the subcritical stationary
solution $(h^*, q^*)$ characterized by
$$
q^* = 2.5.\quad h^*(-3) = 2.
$$
see Figure \ref{fig-stsubSWHc-1}. Tables \ref{table_ss_stsubSWHc-3}-\ref{table_ss_stsubSWHc-5} show the errors and order of convergence of the different methods.
Figures \ref{fig-stsubSWHc-2} and \ref{fig-stsubSWHc-3} show a zoom of the differences between the numerical results obtained with a 100-cell mesh
using WENO$p$, WBWENO$p$, WBWARWENO$p$, $p = 3, 5$ at time $t = 4.$ and the exact solution. \rs{Figure \ref{fig-stsubSWHc-4} shows the numerical solutions obtained for the variable $q$}. \rf{At it can be seen, even though WBWARWENO does not capture the stationary solution with machine accuracy, the results improve those of the standard WENO.}

\begin{table}[ht!]
\centering
\begin{tabular}{|c|c|c|c|c|c|c|c|} \hline
&WB1WENO3  & WBWENO3  & WBMCWENO3 & \multicolumn{2}{|c|}  {WBWARWENO3} & \multicolumn{2}{|c|}  {WENO3}  \\
Cells & Error  & Error   & Error & Error & Order & Error & Order\\
\hline
50  & 0 & 0 &  2.7178E-15 &  4.9069E-2 & - &3.6778E-1 & - \\
100 & 3.1974E-16 & 2.6645E-17 &  2.8110E-15 & 2.3981E-2&  1.03 &   9.3955E-2 &  1.968 \\
200 & 3.8635E-16 &  2.6378E-15& 4.6629E-17 &  4.3491E-3 &  2.46 & 1.3430E-2 &  2.806 \\
400 &  8.8862E-15& 4.6629E-17 &  2.6612E-15&  5.9130E-4 & 2.8787 & 1.7931E-3& 2.904 \\
  \hline
\end{tabular}
\caption{Test \ref{ss_stsubSWHc}. Errors in $L^1$ norm and convergence rates for WB1WENO3,  WBWENO3, WBMCWENO3, WBWARWENO3, and WENO3 at time $t = 4$.} \label{table_ss_stsubSWHc-3}
\end{table}

\begin{table}[ht!]
\centering
\begin{tabular}{|c|c|c|c|c|c|c|c|} \hline
&WB1WENO5& WBWENO5 & WBMCWENO5& \multicolumn{2}{|c|} {WBWARWENO5} &\multicolumn{2}{|c|} {WENO5}  \\
Cells & Error  & Error  & Error & Error & Order  \\
\hline
50  &  0 & 0 & 0 &  3.3234E-2 & - &  4.1777E-1 & - \\
100 &   1.6546E-14 & 2.6645E-17 &  1.2789E-15&  7.5930E-3&  2.129 & 5.1077E-2& 3.031\\
200 &  3.1974E-16&   4.6629E-17&  2.5979E-15& 4.7013E-4 &   4.013& 3.8702E-3 & 3.722\\
400 &  8.6410E-13 & 4.6629E-17&  2.5579E-15 &  2.5026E-05  &  4.231 & 4.18112E-4   & 3.210\\
  \hline
\end{tabular}
\caption{Test \ref{ss_stsubSWHc}. Errors in $L^1$ norm and convergence rates for WB1WENO5,  WBWENO5, WBMCWENO5, WBWARWENO5, and WENO5 at time $t = 4$.} \label{table_ss_stsubSWHc-5}
\end{table}

\begin{figure}
\centering
\includegraphics[width=0.7\textwidth]{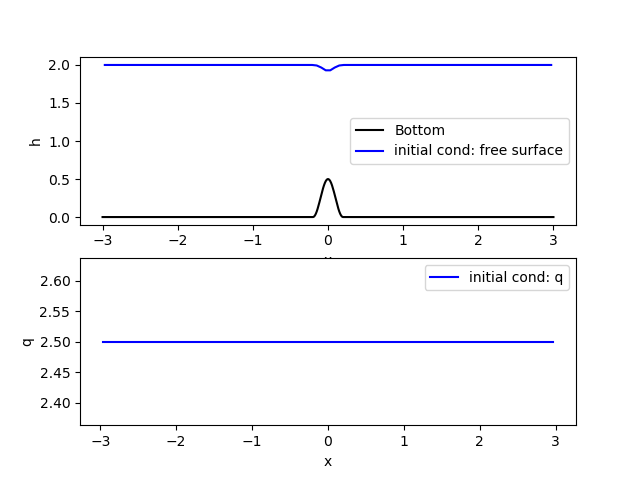}
\caption{Test \ref{ss_stsubSWHc}: initial condition: surface elevation (up) and mass-flow (down)}\label{fig-stsubSWHc-1}
\end{figure}

\begin{figure}
\centering
\includegraphics[width=0.7\textwidth]{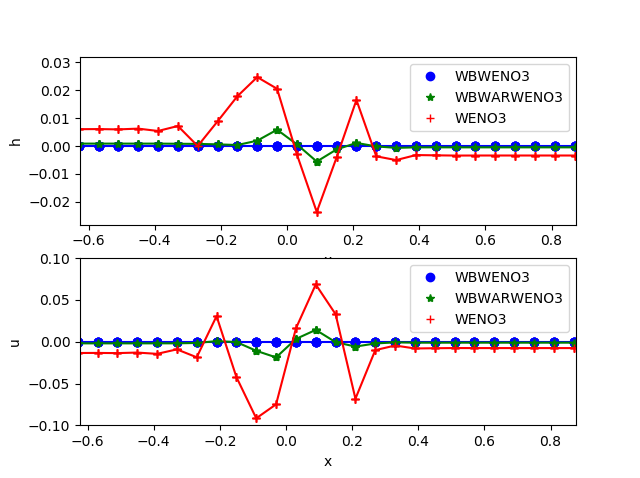}
\caption{Test \ref{ss_stsubSWHc}: Zoom of the differences between the numerical solutions obtained at time $t = 4.$ with WBWENO3, WBWARWEN3, and WENO3 using a mesh of 100 cells and the exact solution: surface elevation (up) and mass-flow (down)}\label{fig-stsubSWHc-2}
\end{figure}

\begin{figure}
\centering
\includegraphics[width=0.7\textwidth]{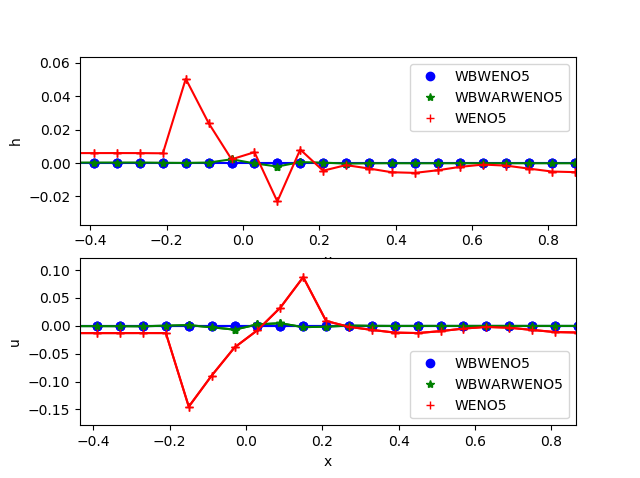}
\caption{Test \ref{ss_stsubSWHc}: Zoom of the differences between the numerical solutions obtained at time $t = 4.$ with WBWENO5, WBWARWEN5, and WENO5 using a mesh of 100 cells and the exact solution: surface elevation (up) and mass-flow (down). }\label{fig-stsubSWHc-3}
\end{figure}

\begin{figure}
\centering
\includegraphics[width=0.45\textwidth]{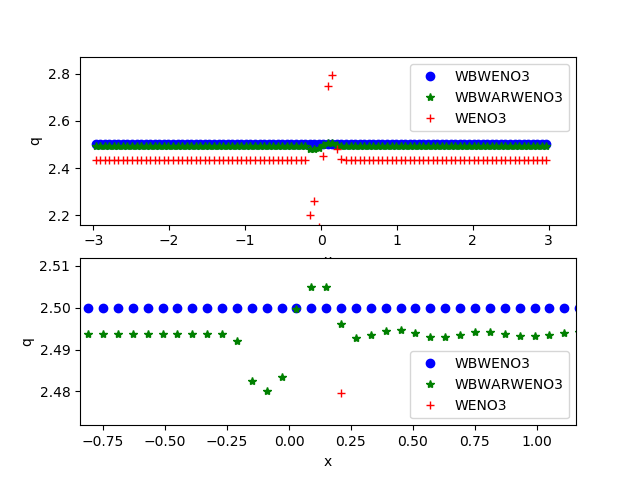}
\includegraphics[width=0.45\textwidth]{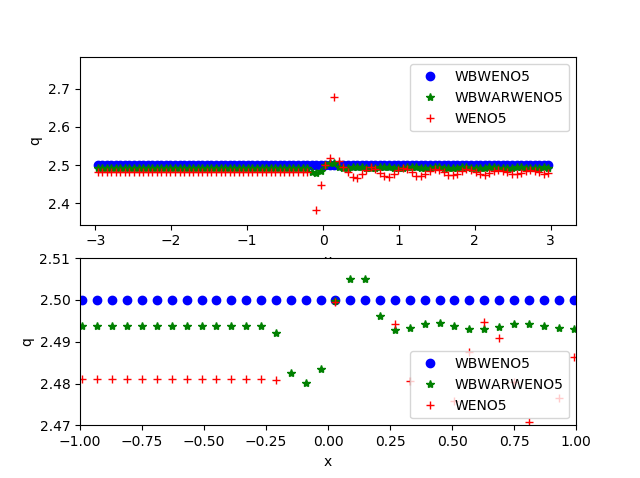}
\caption{\rs{Test \ref{ss_stsubSWHc}: Numerical results for the variable $q$ at $t = 4.$ using a mesh of 100 cells with WBWENO$p$, WBWARWEN$p$, and WENO$p$, $p$ = 3 (left) and $p=5$ (right): general view (up) and zoom close to $x = 0$ (down). } }\label{fig-stsubSWHc-4}
\end{figure}

\subsubsection{Perturbation of a subcritical stationary solution}\label{ss_psubSWHc}
In this test case, we consider an initial condition which is obtained by adding a small perturbation to the stationary solution considered
in the previous case. More precisely, a perturbation of size $\Delta h = 0.02$ is added to the thickness $h$ in the interval
$[-0.4, -0.3]$: see Figure \ref{fig-psubSWHc-1}.  Figures \ref{fig-psubSWHc-2} and \ref{fig-psubSWHc-3} show the difference between the numerical solutions
obtained with WENO$p$, WBWENO$p$, $p = 3,5$ with a 200 point mesh at time $t = 0.15$ and the stationary solution.  A reference solution has been computed with WENO3 in a mesh of 2000 cells.
\rf{Again, the solutions obtained with WBWENO$p$,  WB1WENO$p$, and WBMCWENO$p$, $p = 3,5$ are very close to each other (right figures) and WBWARWENO$p$ gives better results than WENO$p$, $p=3,5$}.

\rf{Before finishing this paragraph, let us compare the mass preservation for the  different third order numerical methods. The total mass water at time $t^n$ is computed by
$$
m_n = \Delta x \sum_i h_i^n.
$$
The numerical experiment is run until $t = 0.3$.
Since the boundary conditions are equal at both extremes of the interval and the simulation is stopped before the waves arrive at the boundaries, the total mass is expected to be preserved.  Table \ref{table_ss_masspsubSWHc} shows the maximum relative deviation of the total mass with respect to its initial value for the different numerical methods, i.e.
$$
\max_n \left| \frac{m_n - m_0}{m_0} \right|
$$
}
\begin{table}[ht!]
\centering
\begin{tabular}{|c|c|c|c|c|} \hline
WENO3 & WBWENO3  & WB1WENO3  & WBWARWENO3 & WBMCWENO3  \\
\hline
4.6319E-15 & 1.3935E-07 & 4.3331E-15 & 4.7813E-15 &  3.4366e-15 \\
  \hline
\end{tabular}
\caption{Test \ref{ss_psubSWHc}:  maximum relative deviation of the total mass.} \label{table_ss_masspsubSWHc}
\end{table}
\rf{According to the discussion in Section \ref{ss:wb1}, WBWENO3 does not preserve the total mass but even in this case the relative deviations are very small.}

\begin{figure}
\centering
\includegraphics[width=0.7\textwidth]{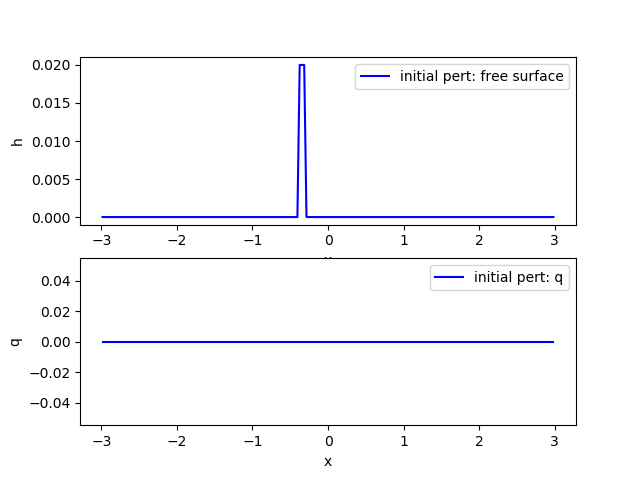}
\caption{Test \ref{ss_psubSWHc}: initial perturbation: surface elevation (up) and mass-flow (down)}\label{fig-psubSWHc-1}
\end{figure}

\begin{figure}
\centering
\includegraphics[width=0.45\textwidth]{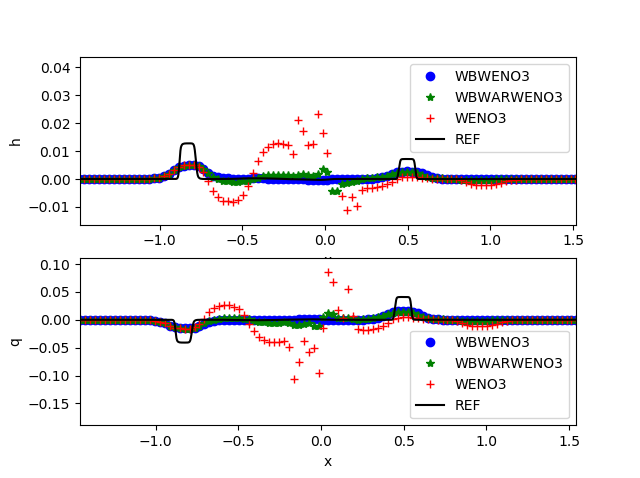}
\includegraphics[width=0.45\textwidth]{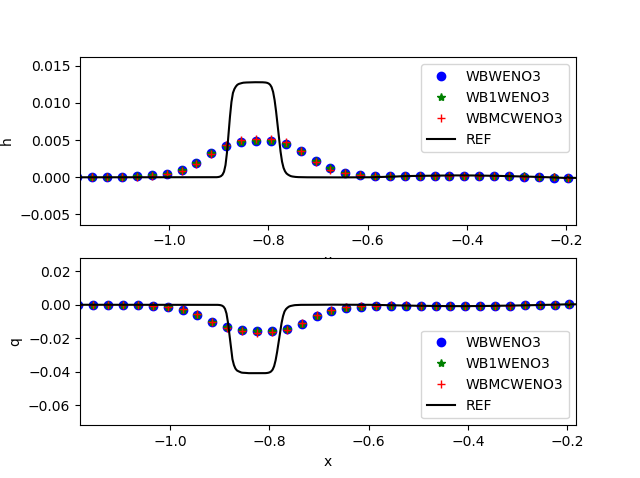}
\caption{Test \ref{ss_psubSWHc}: Zoom of the differences between the numerical solutions obtained at time $t = 0.15$ using a mesh of 200 points and the stationary solution. Left: WBWENO3, WBWARWENO3, WENO3;  surface elevation (left-up) and mass-flow (left-down). 
Right: WBWENO3, WB1WENO3, WBMCWENO3;  surface elevation (right-up) and mass-flow (right-down). }\label{fig-psubSWHc-2}
\end{figure}

\begin{figure}
\centering
\includegraphics[width=0.45\textwidth]{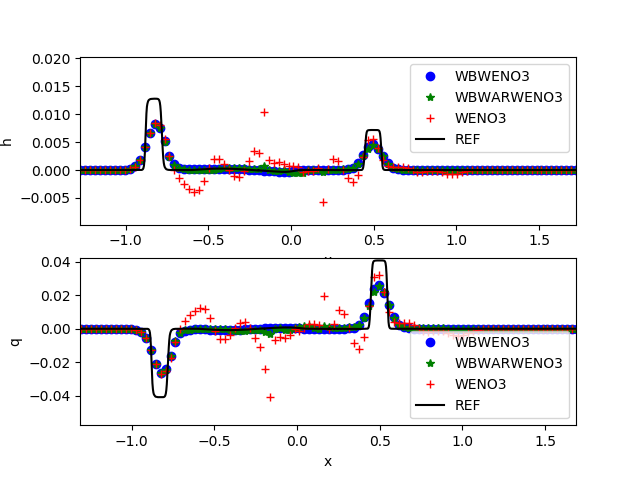}
\includegraphics[width=0.45\textwidth]{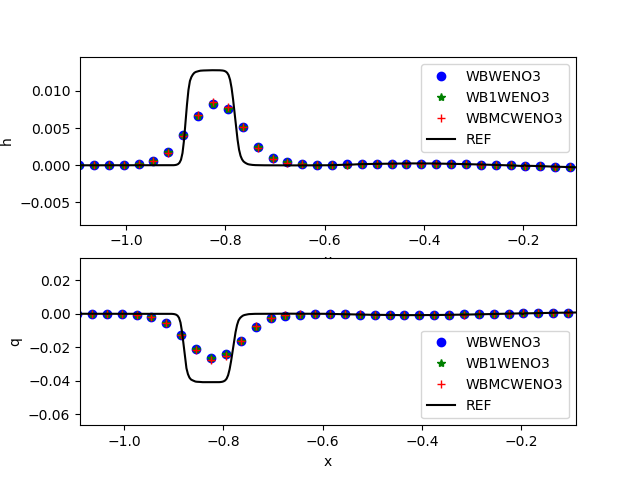}
\caption{Test \ref{ss_psubSWHc}: Zoom of the differences between the numerical solutions obtained at time $t = 0.15$ using a mesh of 200 points and the stationary solution. Left: WBWENO5, WBWARWENO5, WENO5;  surface elevation (left-up) and mass-flow (left-down). 
Right: WBWENO5, WB1WENO5, WBMCWENO5;  surface elevation (right-up) and mass-flow (right-down). }\label{fig-psubSWHc-3}
\end{figure}

\subsubsection{Preservation of a transcritical stationary solution over a discontinuous bottom}\label{ss_stsupSWHd}
We consider now a discontinuous topography given by the depth function
\begin{equation}\label{HSWdisc}
H(x) = \begin{cases} -0.25(1 + \cos(5 \pi (x+1.2))) & \text{if $-1.4 \leq x \leq -1$,}\\
1 & \text{if $x> 0$;}\\
0 & \text{otherwise}.
\end{cases}
\end{equation}

For WENO and WB1WENO methods, the equations for the neighbor nodes of the discontinuity $x_{I-1/2}$ are given by
\eqref{sdmethsstsw} and \eqref{sdmethwb1sw} respectively, with 
\begin{equation}
S_{I-1/2} = S(0.5(U_I + U_{I-1})).
\end{equation}

We take now as initial condition the transcritical admissible stationary solution characterized by:
$$
q^* = 2.5, \quad h^*(0) = \frac{(2.5)^{2/3}}{g^{1/3}}
$$
that  is subcritical at the left of $x = 0$ and supercritical at its right:
see Figure \ref{fig-stsupSWHd-1}. 
\begin{figure}
\centering
\includegraphics[width=0.7\textwidth]{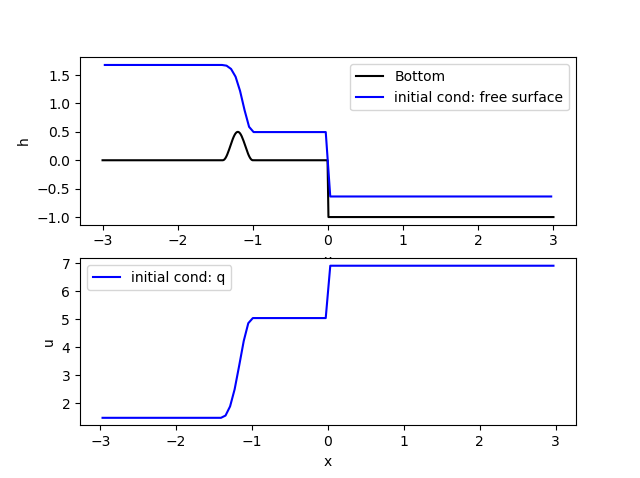}
\caption{Test \ref{ss_stsupSWHd}: initial condition: surface elevation (up) and velocity (down)}\label{fig-stsupSWHd-1}
\end{figure}

In this case, WBMCWENO$p$, $p=3,5$ are unstable.
Figures  \ref{fig-stsupSWHd-2} and \ref{fig-stsupSWHd-3} show the results obtained with WENO$p$ and WBWENO$p$,  $p = 3, 5$ using a mesh of 100 cells
at time $t = 4.$ The numerical results obtained with WB1WENO$p$ and WBMCWENO$p$,  $p = 3,5$ are indistinguishable to those of WBWENO$p$ (they are not plotted).
Table \ref{table_ss_stsupSWHd} shows the error in $L^1$ norm: according to the discussion in Section \ref{s:Hdisc}, WBWENO$p$ and WB1WENO$p$ preserve the stationary solution to machine precision, while
the solutions provided by WENO$p$ are not close to the stationary solution.

\begin{table}[ht!]
\centering
\begin{tabular}{|c|c|c|c|c|c|} \hline
WB1WENO3  & WBWENO3  &  {WENO3} & WB1WENO5  & WBWENO5  & {WENO5} \\
\hline
  7.9602E-16 &  7.9602E-16  & 1.3178 &  7.9602E-16 &   7.9602e-16  & 0.6229   \\
  \hline
\end{tabular}
\caption{Test \ref{ss_stsupSWHd}. Errors in $L^1$ norm for WB1WENO$p$,  WBWENO$p$, and WENO$p$,  $p = 3,5$ at time $t = 4$.} \label{table_ss_stsupSWHd}
\end{table}

\begin{figure}
\centering
\includegraphics[width=0.45\textwidth]{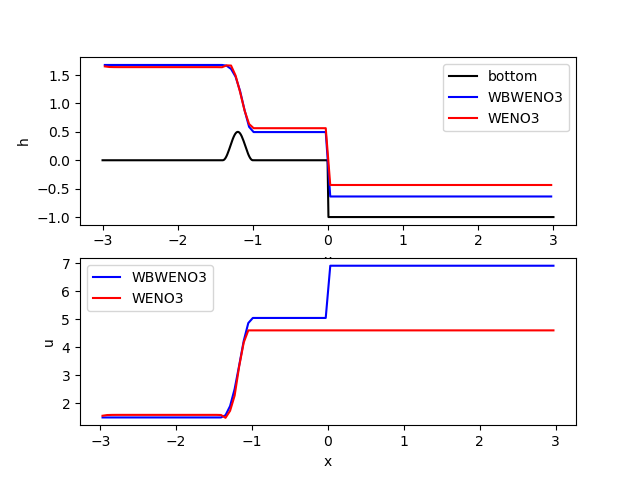}
\includegraphics[width=0.45\textwidth]{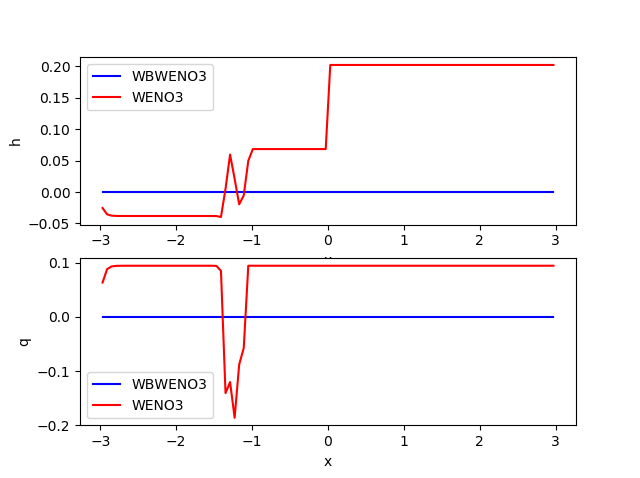}
\caption{Test \ref{ss_stsupSWHd}. Left: Numerical solutions obtained at time $t = 4.$ with WBWENO3, and WENO3 using a mesh of 100 cells: surface elevation (up) and velocity (down). Right: Difference between the numerical solutions and the stationary solution: surface elevation (up) and mass-flow (down) }\label{fig-stsupSWHd-2}
\end{figure}

\begin{figure}
\centering
\includegraphics[width=0.45\textwidth]{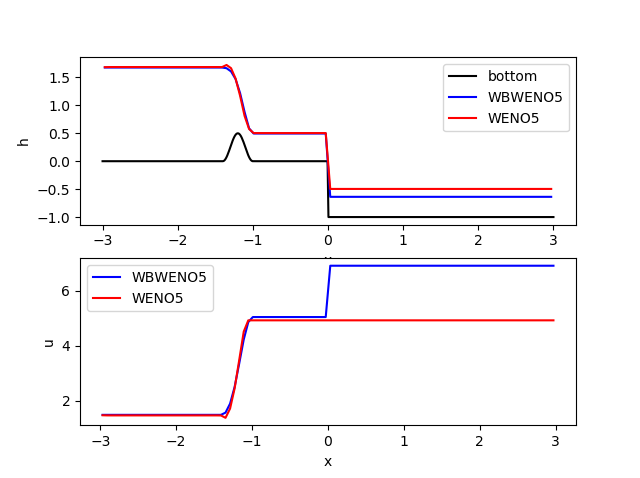}
\includegraphics[width=0.45\textwidth]{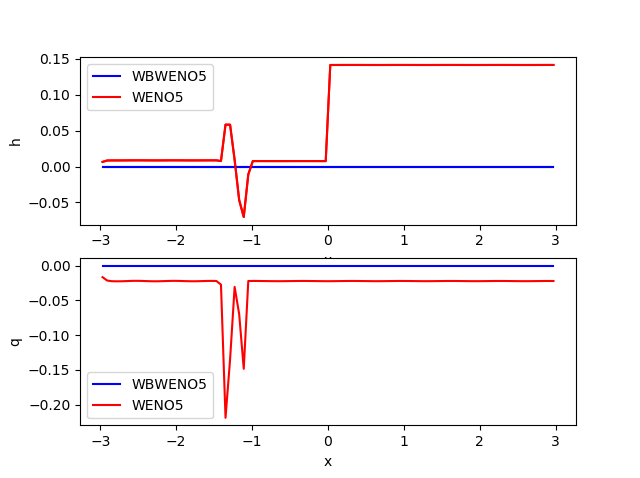}
\caption{Test \ref{ss_stsupSWHd}. Left: Numerical solutions obtained at time $t = 4.$ with WBWENO5, and WENO5 using a mesh of 100 cells: surface elevation (up) and velocity (down). Right: Difference between the numerical solutions and the stationary solution: surface elevation (up) and mass-flow (down)}\label{fig-stsupSWHd-3}
\end{figure}

\subsubsection{Perturbation of a transcritical stationary solution over a discontinuous bottom}\label{ss_ptranSWHd}
We consider now an initial condition which is obtained by adding a small perturbation to the stationary solution considered
in the previous one. More precisely, a perturbation of size $\Delta h = 0.02$ is added to the thickness $h$ in the interval
$[-0.4, -0.3]$: see Figure \ref{fig-psubSWHc-1}.  Figures \ref{fig-ptranSWHd-2} and \ref{fig-ptranSWHd-3} show the difference between the numerical solutions
obtained with WENO$p$, WBWENO$p$, WB1WENO$p$, $p = 3,5$ and the stationary solution. A reference solution has been computed with WBWENO3 in a mesh of 2000 cells. According to the discussion in Section \ref{s:Hdisc}, WB1WENO$p$
can deviate from the stationary solution once the perturbation reaches the discontinuity of $H$. Nevertheless, the differences in this case with WBWENO$p$ are relatively small: see Figure \ref{fig-ptranSWHd-2} (right) and \ref {fig-ptranSWHd-3} (right). The numerical treatment of the source term seems to add some more numerical diffusion; this can be observed clearly for $p = 5$.

\begin{figure}
\centering
\includegraphics[width=0.45\textwidth]{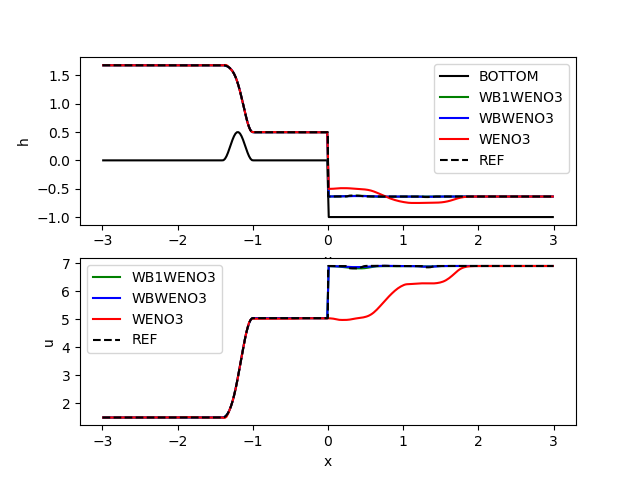}
\includegraphics[width=0.45\textwidth]{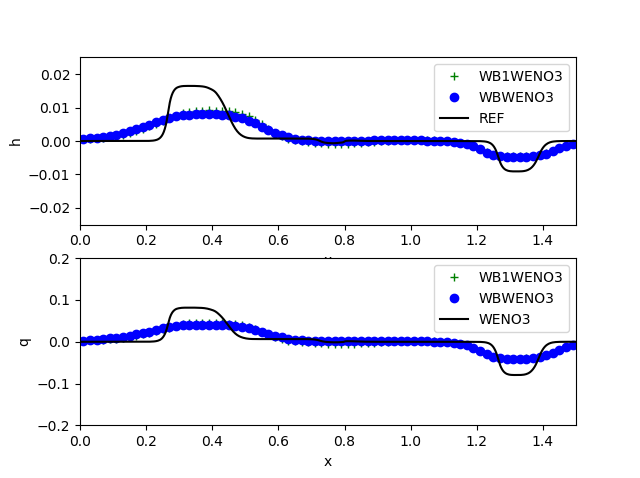}
\caption{Test \ref{ss_ptranSWHd}. Left: Numerical solutions obtained at time $t = 0.2$ with WB1WENO3, WBWENO3, and WENO3 using a mesh of 300 cells: surface elevation (up) and velocity (down). Right: Difference between the numerical solutions obtained with WB1WENO3 and WBWENO3 and the stationary solution: surface elevation (up) and mass-flow (down)}\label{fig-ptranSWHd-2}
\end{figure}

\begin{figure}
\centering
\includegraphics[width=0.45\textwidth]{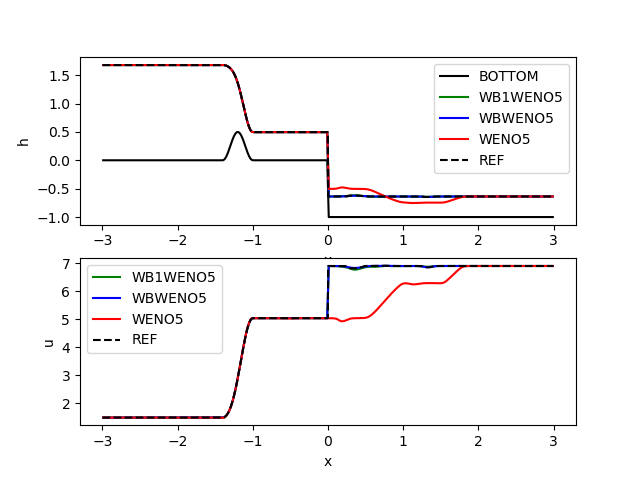}
\includegraphics[width=0.45\textwidth]{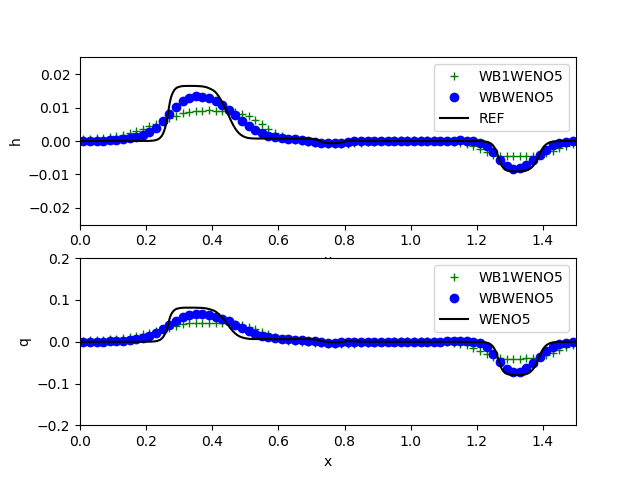}
\caption{Test \ref{ss_ptranSWHd}. Left: Numerical solutions obtained at time $t = 0.2$ with WB1WENO5, WBWENO5, and WENO5 using a mesh of 300 cells: surface elevation (up) and velocity (down). Right: Difference between the numerical solutions obtained with WB1WENO5 and WBWENO5 and the stationary solution: surface elevation (up) and mass-flow (down)}\label{fig-ptranSWHd-3}
\end{figure}

\subsubsection{Mass conservation and computational cost}\label{ss_massSWHc}
In order to measure the mass conservation properties of the different methods and  compare the computational cost, we consider now the depth function
$$
 H(x)  = \begin{cases}  0.13+0.05(x-10)^2 & \text{if  $8 \leq x \leq 12$;}\\
 0.33 & \text{otherwise.}
 \end{cases}
$$
and the initial condition
$$
h_0(x) = h^*(x) + 0.5 \chi_{[5,7]}, \quad q_0(x) = 1,
$$
where $h^*(x)$ is the thickness corresponding to the stationary solution characterized by
$$
q^* = 1, \quad h^*(10) = 1,
$$
and $\chi_{[a,b]}$ denotes the characteristic function of an interval $[a,b]$: see Figure \ref{fig-massSWHc-1}. The computational domain is the interval
$[-10, 30]$ and the boundary conditions are $h = h^*$, $q = 1$ at both extremes. The simulation is run until time $t = 2.5$. Figure \ref{fig-massSWHc-2} shows the results obtained with WBWENO3 and WBMCWENO3 at time $t = 2.5$ using a mesh of 200 cells. In this case, although WBMCWENO3 is not unstable, the results are oscillatory even though the flow is always subcritical. 

The total mass water at time $t^n$ is computed by
$$
m_n = \Delta x \sum_i h_i^n.
$$
Since the boundary conditions are equal at both extremes of the interval and the simulation is stopped before the waves arrive at the boundaries, the total mass is expected to be preserved.  Table \ref{table_ss_massSWHc} shows the maximum relative deviation of the total mass with respect to its initial value for the different third order numerical methods, i.e.
$$
\max_n \left| \frac{m_n - m_0}{m_0} \right|
$$
\begin{table}[ht!]
\centering
\begin{tabular}{|c|c|c|c|c|} \hline
WENO3 & WBWENO3  & WB1WENO3  & WBWARWENO3 & WBMCWENO3  \\
\hline
8.1062E-15 & 9.5985E-06 & 7.3825E-15 & 8.54056E-15 &  7.3825E-15 \\
  \hline
\end{tabular}
\caption{Test \ref{ss_massSWHc}:  maximum relative deviation of the total mass.} \label{table_ss_massSWHc}
\end{table}
According to the discussion in Section \ref{ss:wb1}, WBWENO$p$, $p=3,5$ do not preserve the total mass. Figure \ref{fig-massSWHc-3} shows the variation with time of these deviations for WBWENO3 and WBWENO5: as it can be seen, these variations remain small and get stabilized after some time.
\begin{figure}
\centering
\includegraphics[width=0.7\textwidth]{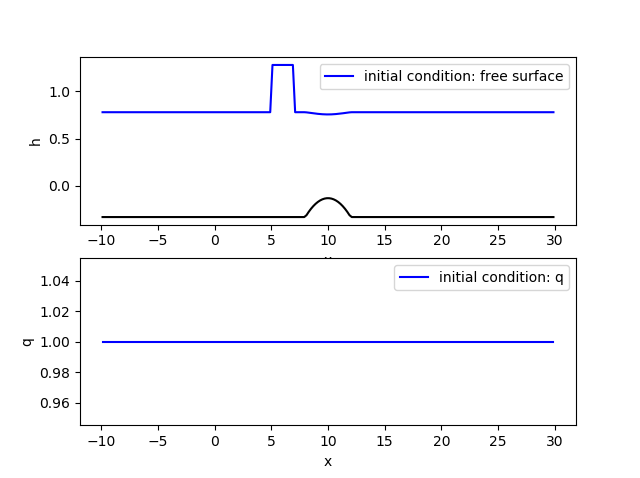}
\caption{Test \ref{ss_massSWHc}: initial condition: surface elevation (up) and mass-flow (down)}\label{fig-massSWHc-1}
\end{figure}

\begin{figure}
\centering
\includegraphics[width=0.7\textwidth]{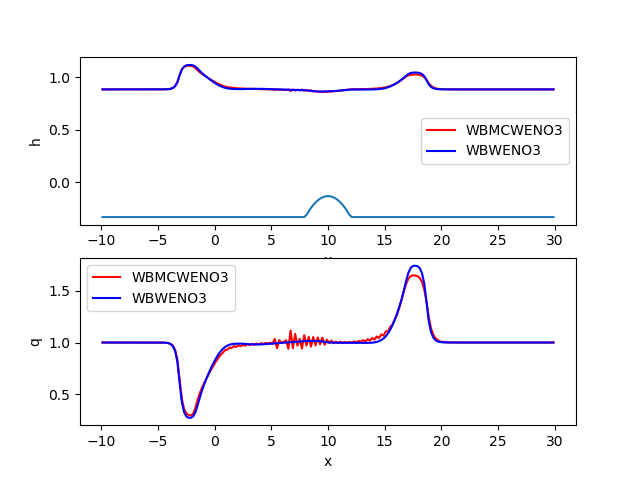}
\caption{Test \ref{ss_massSWHc}: Numerical solutions obtained at time $t = 2.5$ with WBWENO3 and WBMCWENO3 using a mesh of 200 cells: surface elevation (up) and mass-flow (down)}\label{fig-massSWHc-2}
\end{figure}

\begin{figure}
\centering
\includegraphics[width=0.7\textwidth]{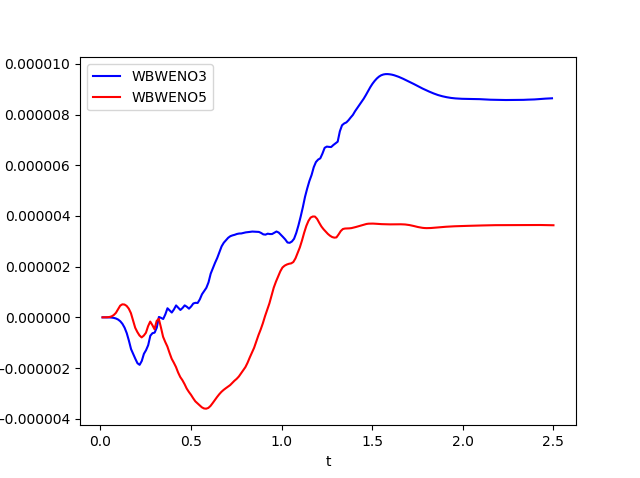}
\caption{Test \ref{ss_massSWHc}: Evolution of the relative deviation of the total mass with time for WBWENO3 and WBWENO5}\label{fig-massSWHc-3}
\end{figure}

Finally, we compare the computational cost of the different numerical methods in this test case. If the only stationary solution to be preserved is computed once before the time loop, the computational  cost of WENO$p$ and WB1WENO$p$ is similar. Table \ref{table_ss_massSWHc_cpu} compares the averaged CPU times of 10 runs using WBWENO$p$, WBWARWENO$p$  and WENO$p$, $p =3,5$. Each row of the table shows the averaged CPU time corresponding to WBWENO$p$ and WBWARWENO$p$ divided by the one corresponding to WENO$p$. Figure \ref{fig-massSWHc-4} shows the graph of the CPU time as a function of 
$\log_2(N)$. As can be seen, the use of the well-balanced technique that allows one to capture every stationary solution increases the CPU time by a factor of
4-5, while the use of the one that only preserves water at rest solutions multiply the computational cost by 1.5 approximately. 
\begin{table}[ht!]
\centering
\begin{tabular}{|c|c|c|c|c|} \hline
Cells & WBWARWENO3  & WBWENO3  &  WBWARWENO5& WBWENO5  \\
\hline
50 & 1.398 & 4.920 & 1.474 & 5.166 \\
100 & 1.453 &  4.695 & 1.534 &   4.895 \\
200 & 1.430  &  4.121 &  1.568 &  4.573\\
400 & 1.354 &  3.604 &  1.492 &  4.138\\
800 & 1.438 & 4.057 &  1.536 & 4.027\\
  \hline
\end{tabular}
\caption{Test \ref{ss_massSWHc}.  CPU times corresponding to WBWARWENO$p$ and WBWENO$p$ divided by the one corresponding to WENO$p$ using meshes of $N = 50$, 100, 200, 400 and 800 cells.} \label{table_ss_massSWHc_cpu}
\end{table}

\begin{figure}
\centering
\includegraphics[width=0.7\textwidth]{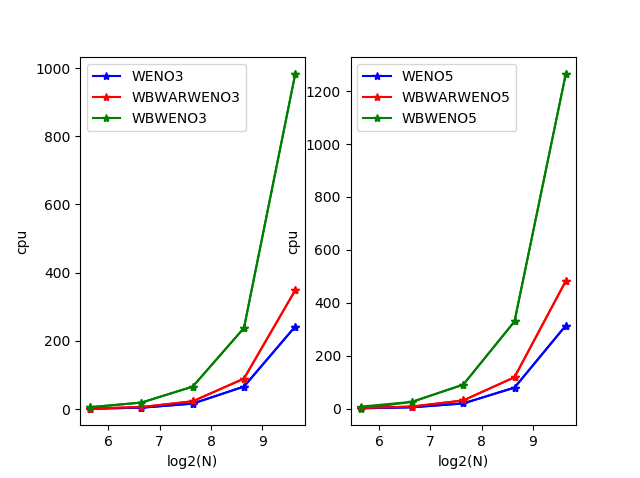}
\caption{Test \ref{ss_massSWHc}: CPU times as a function of $\log_2(N)$ for WENO$P$, WBWARWENO$p$ and WBWENO$p$, $p = 3, 5$}\label{fig-massSWHc-4}
\end{figure}

\section{Conclusions}

Two families of well-balanced high-order finite difference numerical methods have been introduced: one of them preserves every stationary solution while the second preserves a prescribed family of stationary solutions, which can be constituted by a single element. The accuracy and the well-balanced properties of the methods have been analyzed. 

The methods are first introduced in the more simple case in which the source terms do not  involve nonconservative products and the eigenvalues of the Jacobian cannot vanish. Then, the extension of the methods to more complex cases has been discussed: cases in which the stationary solutions are not explicitly known, the bottom is discontinuous, or the eigenvalues can vanish have been addressed. The extension to multidimensional problems is also briefly discussed.

The methods have been applied to a number of numerical tests related to the linear transport equation with linear source term, Burgers' equation with nonlinear source terms, and the shallow water equations. For this latter case, numerical methods that preserve only water  at rest stationary solutions or every stationary solution have been derived. 
A challenging test in which a transcritical solution of the shallow water over a discontinuous bottom is perturbed has been considered. The following conclusions can be drawn:

\begin{itemize}
    \item For test cases in which there is only one known stationary solution involved, when $H$ is continuous, the methods that preserve either all the stationary solutions or only one give essentially the same results. 
    \item When $H$ has discontinuities, both families preserve admissible discontinuous stationary solutions, but the methods that only preserve one stationary solution may fail when this solution is perturbed.
    \item The methods reduce to conservative schemes when the source term vanishes.
    \item  When the system contains conservation laws, the numerical methods are  not in general conservative for them, with the exception of those that preserve only one stationary solution. 
    \item In the particular case of the shallow water model, it has been shown that the numerical methods that preserve water at rest  solutions for the shallow water equations are conservative for the mass equation. Nevertheless we have not been able to find  high-order numerical methods that preserve every stationary solution, that are conservative for the mass equation and stable regardless the regime of the flow:  only partial solutions have been found.
    \item The numerical methods that  preserve a family of stationary solutions are computationally less expensive.
    
\end{itemize}

The derivation of stable high-order finite difference methods that preserve every stationary solution and are conservative for the conservation laws included in the  system remains a challenge, both for the shallow water model and general systems of balance laws.

The main difficulty in applying the methods introduced here to arbitrary systems of balance laws is related to the numerical resolution of a Cauchy problem whose ODE system is not in normal form.  Further extensions would include the application to general systems by solving numerically the Cauchy problems
following the lines discussed in Section \ref{ss:Cauchynum}.

As it has been mentioned in Section \ref{ss:2d}, the strategy introduced here to design numerical methods that preserve only one known stationary solution can be easily extended to multidimensional problems. Nevertheless, the design of well-balanced methods that preserve general stationary solutions for multidimensional problems remains a major challenge.  We hope that the application of numerical methods for solving the Cauchy problems for general 1d problems will give some hints to tackle this challenge, although the boundary value problems to be solved will be related to nonlinear PDEs in that case.

\end{document}